\documentclass[11pt, a4paper, english]{amsart}

\usepackage{amsmath,amssymb,amsthm}
\usepackage{bbm}
\usepackage{newtxtext,newtxmath}
\usepackage[cal=euler,scr=rsfs]{mathalfa}
\usepackage[all]{xy}
\usepackage{here}
\usepackage{tikz}
\usetikzlibrary{arrows.meta}
\usetikzlibrary{decorations.markings}
\usepackage[colorlinks=true,backref=page]{hyperref}
\usepackage{cleveref}

\numberwithin{equation}{section}

\theoremstyle{theorem}
\newtheorem{theorem}{Theorem}[section]
\newtheorem{proposition}[theorem]{Proposition}
\newtheorem{lemma}[theorem]{Lemma}
\newtheorem{corollary}[theorem]{Corollary}

\theoremstyle{definition}
\newtheorem{definition}[theorem]{Definition}
\newtheorem{example}[theorem]{Example}

\theoremstyle{remark}
\newtheorem{remark}[theorem]{Remark}

\newcommand{\Z}{\mathbb{Z}}

\newcommand{\I}{\mathcal{I}}

\SelectTips{cm}{}

\tikzset{myarrow/.style={postaction={decorate},decoration={markings,mark=at position #1 with {\arrow{latex}},}}}

\title[The fundamental group and the MPSS of a directed graph]{The fundamental group and the magnitude-path spectral sequence of a directed graph}

\author{Daisuke Kishimoto}
\address{Faculty of Mathematics, Kyushu University, Fukuoka 819-0395, Japan}
\email{kishimoto@math.kyushu-u.ac.jp}

\author{Yichen Tong}
\address{Institute of Theoretical Sciences, Westlake Institute of Advanced Study, Westlake University, Hangzhou 310030, China}
\email{tongyichen@westlake.edu.cn}

\date{\today}



\begin{document}

\maketitle

\begin{abstract}
  The fundamental group of a directed graph admits a natural sequence of quotient groups called $r$-fundamental groups, and the $r$-fundamental groups can capture properties of a directed graph that the fundamental group cannot capture. The fundamental group of a directed graph is related to path homology through the Hurewicz theorem. The magnitude-path spectral sequence connects magnitude homology and path homology of a directed graph, and it may be thought of as a sequence of homology of a directed graph, including path homology. In this paper, we study relations of the $r$-fundamental groups and the magnitude-path spectral sequence through the Hurewicz theorem and the Seifert-van Kampen theorem.
\end{abstract}


\section{Introduction}\label{Introduction}

The \emph{path homology} of a directed graph was introduced by Grigor'yan, Lin, Muranov, and Yau \cite{GLMY1}, which is sensible to the directions of edges. Path homology is one of the central research objects involving directed graphs, and has been studied in both pure and applied mathematics. There is another homological invariant of directed graphs, called \emph{magnitude homology}, which is defined by Hepworth and Willerton \cite{HW} as the categorification of a numerical invariant of directed graphs, called magnitude. It has also been studied in different contexts independently of path homology. Recently, Asao \cite{A} found an intimate relation between path homology and magnitude homology; he constructed a fourth quadrant spectral sequence called the \emph{magnitude-path spectral sequence} (MPSS), whose $E^1$-page is magnitude homology and an axis of the $E^2$-page is path homology. Hepworth and Roff \cite{HR1} defined yet another homological invariant of directed graphs, called the \emph{reachability homology}, to which the MPSS converges. On the other hand, they \cite{HR2} proved that for each $r\ge 1$, the $E^r$-page of the MPSS satisfies excision with respect to a cofibration of directed graphs defined in \cite{CD}. Then the MPSS can be thought of as a series of homology of a directed graph, including path homology.

Grigor'yan, Lin, Muranov, and Yau \cite{GLMY2} also introduced the fundamental group of a pointed directed graph, denoted by $\pi_1^\mathrm{GLMY}(X,x_0)$; it is defined by using a $C$-homotopy of paths in a directed graph. They also defined a homotopy between maps of directed graphs, and proved that $\pi_1^\mathrm{GLMY}(X,x_0)$ is a homotopy invariant, where this homotopy is a directed analogue of the $A$-homotopy for undirected graphs studied in \cite{At1,At2,BBLL,BKLW,BL}. They also proved the Hurewicz theorem for $\pi_1^\mathrm{GLMY}(X,x_0)$ where the target homology is path homology. On the other hand, Grigor'yan, Jimenez, and Muranov \cite{GJM} introduced the fundamental groupoid of a directed graph, denoted by $\Pi_1^\mathrm{GJM}(X)$; the definition is algebraic, instead of using a $C$-homotopy. They showed the basic properties of the fundamental groupoid of a directed graph such as the Seifert-van Kampen theorem. Recently, Di, Ivanov, Mukoseev, and Zhang \cite{DIMZ} found a chromatic structure in $\Pi_1^\mathrm{GJM}(X)$; they defined the $r$-fundamental groupoid of a directed graph for $1\le r<\infty$ as the edge-path groupoid of a certain simplicial set defined by a directed graph forming a sequence of natural functors
\begin{equation}
  \label{seq}
  \Pi_1^{1}(X)\to\Pi_1^{2}(X)\to\Pi_1^{3}(X)\to\cdots
\end{equation}
which are the identity map on objects and turn out to be surjective on hom-sets, where there is a natural isomorphism $\Pi_1^2(X)\cong\Pi_1^\mathrm{GJM}(X)$. As in Example \ref{square example}, the sequence \eqref{seq} can detect more structures of a directed graph than just the fundamental group alone. We define the \emph{$r$-fundamental group} of a pointed directed graph $(X,x_0)$ by
\[
  \pi_1^{r}(X,x_0)=\Pi_1^{r}(X)(x_0,x_0).
\]
Since there is a natural isomorphism
\begin{equation}
  \label{GLMY pi_1}
  \pi_1^2(X,x_0)\cong\pi_1^\mathrm{GLMY}(X,x_0)
\end{equation}
(see Corollary \ref{2-fundamental group}), \eqref{seq} yields an analogous chromatic structure of $\pi_1^\mathrm{GLMY}(X,x_0)$.

In this paper, we study relations of the MPSS and the $r$-fundamental groupoid, hence the $r$-fundamental group, where we extend the definitions of $r$-fundamental groupoid to $r=\infty$. We prove the Hurewicz theorem for the $r$-fundamental group and the MPSS. Let $E^r_{p,q}(X)$ denote the MPSS for a directed graph $X$, where we use the standard notation for homology spectral sequences, unlikely to the original paper of Asao \cite{A}. Since the MPSS is a fourth quadrant spectral sequence, there is a natural projection $E^r_{1,0}(X)\to E^{r+1}_{1,0}(X)$ for $2\le r<\infty$. We say that a directed graph is \emph{connected} if any two vertices are connected by a path, i.e. a zig-zag of finitely many edges.

\begin{theorem}
  \label{main 1}
  Let $(X,x_0)$ be a connected pointed directed graph. Then there is a commutative diagram
  \begin{equation}
    \label{ladder h}
    \xymatrix{
      \pi_1^{2}(X,x_0)\ar[r]\ar[d]^{h^{2}}&\pi_1^3(X,x_0)\ar[r]\ar[d]^{h^3}&\pi_1^4(X,x_0)\ar[d]^{h^4}\ar[r]&\cdots\\
      E^2_{1,0}(X)\ar[r]&E^3_{1,0}(X)\ar[r]&E^4_{1,0}(X)\ar[r]&\cdots
    }
  \end{equation}
  such that each $h^{r}$ is identified with abelianization, where the top sequence is induced from \eqref{seq} and the bottom sequence consists of natural projections. Moreover, there is a natural map
  \[
    h^{\infty}\colon\pi_1^{\infty}(X,x_0)\to\mathrm{RH}_1(X)
  \]
  which is identified with abelianization, where $\mathrm{RH}_*(X)$ denotes the reachability homology.
\end{theorem}

\begin{remark}
  Let $(X,x_0)$ be a pointed directed graph. As mentioned above, $E^2_{1,0}(X)$ is isomorphic to the first path homology of $X$. Then by \eqref{GLMY pi_1}, Theorem \ref{main 1}  recovers the above mentioned Hurewicz theorem for $\pi_1^\mathrm{GLMY}(X,x_0)$ proved in \cite{GLMY1}.
\end{remark}

We also prove the Seifert-van Kampen theorem for the $r$-fundamental groupoids of an \emph{$r$-separable} pair of directed graphs for $1\le r\le\infty$. The $2$-separability is weaker than the condition assumed for the Seifert-van Kampen theorem for $\Pi_1^\mathrm{GJM}(X)$ in \cite{GJM}, so our result refines it. To prove the Seifert-van Kampen theorem, we need a combinatorial description of the $r$-fundamental groupoid, that is, a description in terms of the homotopy classes of paths in a directed graph by combinatorially given homotopies between paths. The homotopy of paths is given by using a certain graph $\Gamma_r$ (see Section \ref{Combinatorial description}) which detects a nontrivial differential on the $E^r$-page of the MPSS (Proposition \ref{differential}).

\begin{theorem}
  \label{main 2}
  Let $(X,Y)$ be an $r$-separable pair of directed graphs. Then the commutative diagram
  \[
    \xymatrix{
      \Pi_1^{r}(X\cap Y)\ar[r]\ar[d]&\Pi_1^{r}(X)\ar[d]\\
      \Pi_1^{r}(Y)\ar[r]&\Pi_1^{r}(X\cup Y)
    }
  \]
  is a pushout of groupoids.
\end{theorem}

As a corollary to Theorem \ref{main 2}, the Seifert-van Kampen theorem for the $r$-fundamental groups of an $r$-separable pair is obtained (Corollary \ref{van Kampen pi_1}). We also get the following Mayer-Vietoris sequence for $E^r_{1,0}$ of the MPSS for any $1\le r<\infty$. For directed graphs $X$ and $Y$, let $i_A\colon X\cap Y\to A$ and $j_A\colon A\to X\cup Y$ denote inclusions for $A=X,Y$.

\begin{corollary}
  \label{Mayer-vietoris sequence}
  Let $(X,Y)$ be an $r$-separable pair of directed graphs with $2\le r<\infty$. If $X,Y,X\cap Y$ are connected, then the sequence
  \[
    E_{1,0}^r(X\cap Y)\xrightarrow{((i_X)_*,(i_Y)_*)}E_{1,0}^r(X)\oplus E_{1,0}^r(Y)\xrightarrow{(j_X)_*-(j_Y)_*}E_{1,0}^r(X\cup Y)\to 0
  \]
  is exact.
\end{corollary}

Consider the pushout of directed graphs
\begin{equation}
  \label{pushout phi}
  \xymatrix{
    A\ar[r]^f\ar[d]&X\ar[d]\\
    Y\ar[r]&X\cup_AY.
  }
\end{equation}
In \cite{CD}, Carranza \textit{et al.} proved that the category of directed graphs carries a cofibration category structure where weak equivalences are maps inducing isomorphisms in path homology. Hepworth and Roff \cite{HR1} defined the relative MPSS, and proved that if the map $f\colon A\to X$ is a cofibration of Carranza \textit{et al.} \cite{CD}, then the natural map
\[
  E^r_{p,q}(X,A)\to E^r_{p,q}(X\cup_AY,Y)
\]
is an isomorphism for any $r\ge 1$, which is excision of the MPSS. They also proved the Mayer-Vietoris sequence of the MPSS of $X\cup_AY$ for $r=1,2$, and asked whether there is the Mayer-Vietoris sequence for $3\le r<\infty$. On the other hand, Hepworth and Roff \cite{HR2} introduced a long cofibration by relaxing the definition of a cofibration, and proved that reachability homology enjoys excision and the Mayer-Vietoris sequence for the pushout \eqref{pushout phi} whenever the map $f\colon A\to X$ is a long cofibration. We define an \emph{$r$-cofibration} of directed graphs for $1\le r\le\infty$ by relaxing the definition of a cofibration of Carranza \textit{et al.} \cite{CD} such that a $1$-cofibration and an $\infty$-cofibration are exactly a cofibration of Carranza \textit{et al.} and a long cofibration, respectively. Then as an application to Theorem \ref{main 2}, we prove the Seifert-van Kampen theorem for the pushout \eqref{pushout phi} with $f\colon A\to X$ an $r$-cofibration (Corollary \ref{van Kampen pushout}). From this corollary, we deduce the Mayer-Vietoris sequence of the $(1,0)$-block of the MPSS of \eqref{pushout phi} for any $1\le r<\infty$ (Corollary \ref{Mayer-Vietoris sequence pushout}), which is a partial answer to the above mentioned Hepworth and Roff's question \cite{HR1}.

\subsection*{Acknowledgement}

The authors are grateful to Yasuhiko Asao and Luigi Caputi for discussions. They are also grateful to the referee for useful comments that substantially simplified the proof of Theorem \ref{main 1}. The first author was partially supported by JSPS KAKENHI Grant Number JP22K03284, and the second author was partially supported by JST SPRING Grant Number JPMJSP2110.


\section{$r$-Fundamental groupoid}\label{r-Fundamental groupoid}

In this section, we recall the $r$-fundamental groupoid of a directed graph defined by Di, Ivanov, Mukoseev, and Zhang \cite{DIMZ} and its algebraic description.


\subsection{Directed graph}

We set notation and terminology for directed graphs.

\begin{definition}
  A \emph{directed graph} $X$ consists of the vertex set $V(X)$ and the edge set $E(X)\subset V(X)\times V(X)-\Delta$, where $\Delta$ denotes the diagonal set and $(x,y)\in E(X)$ means an edge directed from $x$ to $y$.
\end{definition}

Note that we do not allow a directed graph to have loop edges and multiple edges with the same direction. Let $X$ be a directed graph. A \emph{subgraph} of $X$ is a directed graph whose vertex and edge sets are subsets of $V(X)$ and $E(X)$, respectively. The \emph{induced subgraph} of a directed graph $X$ over $W\subset V(X)$, denoted $X_W$, is a subgraph of $X$ such that $E(X_W)=W$ and
\[
  E(X_W)=\{(x,y)\in E(X)\mid x,y\in W\}.
\]
We define the underlying undirected graph $\widehat{X}$ of $X$ by forgetting about directions of edges. In particular, $\widehat{X}$ has multiple edges between vertices $x$ and $y$ if and only if $(x,y),(y,x)\in E(X)$.

\begin{definition}
  A \emph{map} $f\colon X\to Y$ between directed graphs $X$ and $Y$ is a map $f\colon V(X)\to V(Y)$ such that for any $(x,y)\in V(X)$, either $(f(x),f(y))\in E(Y)$ or $f(x)=f(y)$ holds.
\end{definition}

Recall that a \emph{homomorphism} $f\colon X\to Y$ between directed graphs $X$ and $Y$ is a map such that for any $(x,y)\in E(X)$, $(f(x),f(y))\in E(Y)$ holds. Then homomorphisms are maps, but maps are not necessarily homomorphisms. We define the box product and the strong product of directed graphs.

\begin{definition}
  The \emph{box product} $X\square Y$ of directed graphs $X$ and $Y$ is the directed graph such that $V(X\square Y)=V(X)\times V(Y)$ and $((x_0,y_0),(x_1,y_1))\in E(X\square Y)$ if either of the following conditions holds.
  \begin{enumerate}
    \item $x_0=x_1$ and $(y_0,y_1)\in E(Y)$.

    \item $(x_0,x_1)\in E(X)$ and $y_0=y_1$.
  \end{enumerate}
\end{definition}

The \emph{strong product} $X\boxtimes Y$ of directed graphs $X$ and $Y$ is defined by adding edges $((x_0,x_1),(y_0,y_1))$ with $(x_0,x_1)\in E(X)$ and $(y_0,y_1)\in E(Y)$ to $X\square Y$. Note that the strong product is the product in the category of directed graph, while the box product just defines a monoidal structure on this category.


\subsection{Nerve}

Let $X$ be a directed graph. For vertices $x$ and $y$ of $X$, we define $d(x,y)$ to be the minimal integer $n$ such that there is a sequence of vertices $x=x_0,x_1,\ldots,x_n=y$ of $X$ where $(x_i,x_{i+1})$ is an edge of $X$ for $i=0,1,\ldots,n-1$. If no such an integer exists, then we set $d(x,y)=\infty$. Observe that $d(x,y)=0$ if and only if $x=y$, and the triangle inequality
\begin{equation}
  \label{triangle inequality}
  d(x,y)+d(y,z)\ge d(x,z)
\end{equation}
holds. However, $d(x,y)\ne d(y,x)$ in general. Then $d(x,y)$ is a quasimetric on $X$. For $(x_0,\ldots,x_n)\in V(X)^{n+1}$, we define
\[
  L(x_0,\ldots,x_n)=\sum_{i=0}^{n-1}d(x_i,x_{i+1}).
\]

\begin{definition}
  The \emph{nerve} $\mathrm{N}(X)$ of a directed graph $X$ is defined to be the simplicial set such that
  \[
    \mathrm{N}_n(X)=\{(x_0,\ldots,x_n)\in V(X)^{n+1}\mid L(x_0,\ldots,x_n)<\infty\}
  \]
  and the face maps and the degeneracy maps are given by
  \begin{align*}
    d_i(x_0,\ldots,x_n)&=(x_0,\ldots,\widehat{x_i},\ldots,x_n)\\
    s_i(x_0,\ldots,x_n)&=(x_0,\ldots,x_i,x_i,\ldots,x_n)
  \end{align*}
\end{definition}

Observe that
\begin{equation}
  \label{L}
  L(d_i(\sigma))\le L(\sigma)\quad\text{and}\quad L(s_i(\sigma))=L(\sigma)
\end{equation}
for $\sigma\in\mathrm{N}_n(X)$. Then the nerve of a directed graph is well-defined. A directed graph $X$ defines a preorder $P$ such that $P=V(X)$ as a set and $x<y$ for $x,y\in P$ if $d(x,y)<\infty$. As in \cite[Proposition 1.1]{DIMZ}, it is obvious that the nerve of a directed graph $X$ coincides with the nerve of the preorder $P$ considered as a category.

For $0\le r<\infty$, we define the simplicial set $\mathrm{N}^r(X)$ for a directed graph $X$ by
\[
  \mathrm{N}_n^r(X)=\{(x_0,\ldots,x_n)\in V(X)^{n+1}\mid L(x_0,\ldots,x_n)\le r\}
\]
and the face maps and the degeneracy maps are the restriction of those of $\mathrm{N}(X)$. By \eqref{L}, $\mathrm{N}^r(X)$ is a well-defined simplicial set, and we get a sequence of simplicial sets
\[
  \mathrm{N}^0(X)\subset\mathrm{N}^1(X)\subset\mathrm{N}^2(X)\subset\cdots\subset\mathrm{N}(X).
\]

\begin{definition}
  For $1\le r\le\infty$, the \emph{$r$-fundamental groupoid} of a directed graph $X$, denoted by $\Pi_1^r(X)$, is defined to be the edge-path groupoid of $\mathrm{N}^r(X)$, where $\mathrm{N}^\infty(X)=\mathrm{N}(X)$.
\end{definition}

The $1$-fundamental groupoid can be easily computed.

\begin{proposition}
  \label{1-fundamental groupoid}
  For a directed graph $X$, $\Pi_1^1(X)$ is naturally isomorphic to the edge-path groupoid of the underlying undirected graph of $X$.
\end{proposition}

\begin{proof}
  By definition, $\mathrm{N}^1(X)$ is identified with the underlying undirected graph of $X$, so the statement follows.
\end{proof}

Then the $1$-fundamental groupoid forgets about directions of edges of a directed graph. But the following example shows that the $r$-fundamental groupoid for $2\le r\le\infty$ is sensible to directions of edges. The example also shows that the sequence of the $r$-fundamental groupoids is certainly a stronger invariant than just one $r$-fundamental groupoid alone.

\begin{example}
  \label{square example}
  Consider the following directed graphs $X_1,X_2,X_3,X_4$.

  \begin{figure}[htbp]
    \centering
    \begin{tikzpicture}[x=5mm, y=5mm, thick]
      \draw[myarrow=.6](0,0)--(0,3);
      \draw[myarrow=.6](0,3)--(3,3);
      \draw[myarrow=.6](3,0)--(3,3);
      \draw[myarrow=.6](0,0)--(3,0);
      \fill[black](0,0) circle(2pt);
      \fill[black](3,0) circle(2pt);
      \fill[black](3,3) circle(2pt);
      \fill[black](0,3) node[left]{$x_1$} circle(2pt);
      \draw(1.5,3) node[above=1mm]{$X_1$};
      \draw[myarrow=.6](6,0)--(6,3);
      \draw[myarrow=.6](6,3)--(9,3);
      \draw[myarrow=.6](9,3)--(9,0);
      \draw[myarrow=.6](6,0)--(9,0);
      \fill[black](6,0) circle(2pt);
      \fill[black](9,0) circle(2pt);
      \fill[black](9,3) circle(2pt);
      \fill[black](6,3) node[left]{$x_2$} circle(2pt);
      \draw(7.5,3) node[above=1mm]{$X_2$};
      \draw[myarrow=.6](12,0)--(12,3);
      \draw[myarrow=.6](12,3)--(15,3);
      \draw[myarrow=.6](15,3)--(15,0);
      \draw[myarrow=.6](15,0)--(12,0);
      \fill[black](12,0) circle(2pt);
      \fill[black](15,0) circle(2pt);
      \fill[black](15,3) circle(2pt);
      \fill[black](12,3) node[left]{$x_3$} circle(2pt);
      \draw(13.5,3) node[above=1mm]{$X_3$};
      \draw[myarrow=.6](18,3)--(18,0);
      \draw[myarrow=.6](18,3)--(21,3);
      \draw[myarrow=.6](21,0)--(21,3);
      \draw[myarrow=.6](21,0)--(18,0);
      \fill[black](18,0) circle(2pt);
      \fill[black](21,0) circle(2pt);
      \fill[black](21,3) circle(2pt);
      \fill[black](18,3) node[left]{$x_4$} circle(2pt);
      \draw(19.5,3) node[above=1mm]{$X_4$};
    \end{tikzpicture}
    \caption{The directed graphs $X_1,X_2,X_3,X_4$}
  \end{figure}

  \noindent Then the hom-sets $\Pi_1^{r}(X_i)(x_i,x_i)$ for $i,r=1,2,3,4$ are groups given in the following table. Then the sequences $\Pi_1^1(X_i),\ldots,\Pi_1^4(X_i)$ distinguish $X_1,\ldots,X_4$ though just one $\Pi_1^r(X_i)$ alone for $r=1,2,3,4$ cannot.

  \begin{table}[htbp]
    \renewcommand{\arraystretch}{1.2}
    \begin{tabular}{|c|c|c|c|c|}
      \hline
      &$r=1$&$r=2$&$r=3$&$r=4$\\
      \hline
      $X_1$&$\Z$&$0$&$0$&$0$\\
      $X_2$&$\Z$&$\Z$&$0$&$0$\\
      $X_3$&$\Z$&$\Z$&$\Z$&$0$\\
      $X_4$&$\Z$&$\Z$&$\Z$&$\Z$\\
      \hline
    \end{tabular}
  \end{table}
\end{example}

We consider how the $\infty$-fundamental groupoid is recovered from the $r$-fundamental groupoids for $1\le r<\infty$.

\begin{lemma}
  \label{realization}
  There is a natural isomorphism between the edge-path groupoid of a simplicial set $X$ and the full subcategory of the fundamental groupoid of the geometric realization $|X|$ with objects $X_0$.
\end{lemma}

\begin{proof}
  The proof for a simplicial complex in \cite[16 Theorem, p. 137]{S} works verbatim for a simplicial set.
\end{proof}

\begin{lemma}
  \label{colim}
  For a directed graph $X$, there is a natural isomorphism
  \[
    \Pi_1^\infty(X)\cong\underset{r\to\infty}{\mathrm{colim}}\,\Pi_1^r(X).
  \]
\end{lemma}

\begin{proof}
  By definition, we have
  \[
    \mathrm{N}(X)=\underset{r\to\infty}{\mathrm{colim}}\,\mathrm{N}^r(X).
  \]
  Then by Lemma \ref{realization}, there are natural isomorphisms
  \[
    \Pi_1^\infty(X)\cong\Pi_1^\infty(|\mathrm{N}(X)|)\cong\underset{r\to\infty}{\mathrm{colim}}\,\Pi_1(|\mathrm{N}^r(X)|)\cong\underset{r\to\infty}{\mathrm{colim}}\,\Pi_1^r(X).
  \]
\end{proof}


\subsection{Algebraic description}

We recall the algebraic description of the $r$-fundamental groupoid given in \cite{DIMZ}. Let $X$ be a directed graph. Let
\[
  \widetilde{E}(X)=E(X)\cup\Delta
\]
where $\Delta$ denotes the diagonal set of $V(X)\times V(X)$. We define maps
\[
  s,t\colon\widetilde{E}(X)\to V(X)
\]
by $s(x,y)=x$ and $t(x,y)=y$.

\begin{definition}
  \label{G(X)}
  Let $X$ be a directed graph. For $1\le r<\infty$, we define the groupoid $\mathcal{G}^{r}(X)$ as follows.

  \begin{enumerate}
    \item Objects are vertices of $X$.

    \item each element $e\in\widetilde{E}(X)$ determines morphisms $e\colon s(e)\to t(e)$ and $e^{-1}\colon t(e)\to s(e)$ satisfying
    \[
      e\circ e^{-1}=1_{t(e)}\quad\text{and}\quad e^{-1}\circ e=1_{s(e)}.
    \]

    \item For any $x\in V(X)$, the morphisms $(x,x),(x,x)^{-1}\colon x\to x$ are the identity.

    \item Every morphism is the composite of finitely many morphisms given by elements of $\widetilde{E}(X)$.

    \item For $1\le p,q\le r$, let $e_1,\ldots,e_p,f_1,\ldots,f_q$ be any elements of $\widetilde{E}(X)$ such that $e_1\circ\cdots\circ e_p$ and $f_1\circ\cdots\circ f_q$ are defined. If $t(e_1)=t(f_1)$ and $s(e_p)=s(f_q)$, then
    \[
      e_1\circ\cdots\circ e_p=f_1\circ\cdots\circ f_q.
    \]
  \end{enumerate}
\end{definition}

\begin{remark}
  The conditions (1) to (4) gives a free groupoid generated by $\widetilde{E}(X)$, and that (5) is the only condition which is related to $1\le r<\infty$.
\end{remark}

\begin{proposition}
  [Di, Ivanov, Mukoseev, and Zhang {\cite[Proposition 2.2]{DIMZ}}]
  \label{algebric description}
  For a directed graph $X$ and $1\le r<\infty$, there is a natural isomorphism
  \[
    \Pi_1^r(X)\cong\mathcal{G}^r(X).
  \]
\end{proposition}

Recall that the fundamental groupoid of a directed graph $X$ defined by Grigor'yan, Jimenez, and Muranov \cite{GJM} is denoted by $\Pi_1^\mathrm{GJM}(X)$.

\begin{corollary}
  For a directed graph $X$, there is a natural isomorphism
  \[
    \Pi_1^2(X)\cong\Pi_1^\mathrm{GJM}(X).
  \]
\end{corollary}

\begin{proof}
  It is shown in \cite[Theorem 2.4]{GJM} that $\mathcal{G}^2(X)$ is naturally isomorphic to $\Pi_1^\mathrm{GJM}(X)$. Then the statement follows from Proposition \ref{algebric description}.
\end{proof}

\begin{corollary}
  For a directed graph $X$, each functor in the canonical sequence
  \begin{equation}
    \label{sequence Pi}
    \Pi_1^1(X)\to\Pi_1^2(X)\to\Pi_1^3(X)\to\cdots
  \end{equation}
  is the identity map on objects and surjective on hom-sets.
\end{corollary}

\begin{proof}
  There is a functor $\mathcal{G}^r(X)\to\mathcal{G}^{r+1}(X)$ which is the identity map on objects and sends a morphism of $\mathcal{G}^r(X)$ represented by $e\in\widetilde{E}(X)$ to that of $\mathcal{G}^{r+1}(X)$ represented by $e$. Clearly, this functor is surjective on hom-sets, and the proof of Proposition \ref{algebric description} implies that it is identified with the canonical functor $\Pi_1^r(X)\to\Pi_1^{r+1}(X)$ induced from the inclusion $\mathrm{N}^r(X)\to\mathrm{N}^{r+1}(X)$. Thus the statement is proved.
\end{proof}

By Lemma \ref{colim} and Theorem \ref{equivalence def}, we get:

\begin{corollary}
  For a directed graph $X$, there is a natural isomorphism
  \[
    \Pi_1^\infty(X)\cong\underset{r\to\infty}{\mathrm{colim}}\,\mathcal{G}^r(X).
  \]
\end{corollary}

\begin{remark}
  The groupoid $\underset{r\to\infty}{\mathrm{colim}}\,\mathcal{G}^r(X)$ is explicitly given as follows. Let $\mathcal{F}(X)$ denote the groupoid defined by the conditions (1), (2), (3) and (4) of Definition \ref{G(X)}. Then $\underset{r\to\infty}{\mathrm{colim}}\,\mathcal{G}^r(X)$ is the groupoid obtained by quotienting out $\mathcal{F}(X)$ by the condition (5) of Definition \ref{G(X)} for all $1\le r<\infty$.
\end{remark}


\section{Combinatorial description}\label{Combinatorial description}

In this section, we give a combinatorial description of the $r$-fundamental groupoid, and as its applications, we show the homotopy invariance and the product formula of the $r$-fundamental groupoid.


\subsection{Path}

We define a path in a directed graph and a $C_r$-homotopy between two paths. Let $[n]=\{0,1,\ldots,n\}$ for $n\ge 0$. Let $\I_0$ be a singleton consisting of a directed graph with a single vertex $0$. For $n\ge 1$, let $\I_n$ denote the set of all directed graphs with vertex set $[n]$ having exactly one edge $(i,i+1)$ or $(i+1,i)$ for each $i=0,1,\ldots,n-1$ and no other edges. Namely, $\I_n$ is the set of directed graphs whose underlying undirected graphs are a path graph with $n$ edges.

\begin{definition}
  A \emph{path} in a directed graph $X$ is a map $I_n\to X$ for some $I_n\in\I_n$ with $n\ge 0$.
\end{definition}

We say that a path in a directed graph is \emph{reduced} if it is given by a homomorphism. Observe that every reduced path is determined by its end points and image. Then we often identify a reduced path with its image. In particular, we often consider an edge of a directed graph as a reduced path. We say that a path is a \emph{loop} if its initial and terminal points are the same. Let $\vec{I}_n$ denote the special element of $\I_n$ such that
\[
  E(\vec{I}_n)=\{(i,i+1)\mid i=0,1,\ldots,n-1\}.
\]

\begin{definition}
  A \emph{directed path} in a directed graph $X$ is a path $\vec{I}_n\to X$.
\end{definition}

Let $I_m\in\I_m$ and $I_n\in\I_n$. We define an element $I_m\cdot I_n$ of $\I_{m+n}$ by
\[
  E(I_m\cdot I_n)=E(I_m)\cup\{(x+m,y+m)\mid(x,y)\in E(I_n)\}.
\]
Let $f\colon I_m\to X$ and $g\colon I_n\to X$ be paths in a directed graph $X$ satisfying $f(m)=g(0)$. We define the concatenation of paths $f$ and $g$ by the path $f\cdot g\colon I_m\cdot I_n\to X$ such that
\[
  (f\cdot g)(i)=
  \begin{cases}
    f(i)&i=0,1,\ldots,m\\
    g(i-m)&i=m,m+1,\ldots,m+n.
  \end{cases}
\]
Clearly, the concatenation $f\cdot g$ is a well-defined path in $X$. For a vertex $x$ of a directed graph $X$, let $c_x\colon I_0\to X$ be a map such that $c_x(0)=x$. The following properties of the concatenation of paths are immediate from the definition.

\begin{lemma}
  \label{concatenation}
  Let $f,g,h$ be paths in a directed graph $X$, and let $\varphi\colon X\to Y$ be a map of directed graphs. Then the following holds.

  \begin{enumerate}
    \item $c_x\cdot f=f=f\cdot c_y$, where $f$ is a path from $x$ to $y$.

    \item $f\cdot(g\cdot h)=(f\cdot g)\cdot h$.

    \item $\varphi\circ(f\cdot g)=(\varphi\circ f)\cdot(\varphi\circ g)$.
  \end{enumerate}
\end{lemma}

We define a $C_r$-homotopy of paths in a directed graph for $0\le r<\infty$. Let $\Gamma_0$ be the directed graph with a single vertex $u_0=v_0$, and let $\Gamma_1$ be the directed graph with $V(\Gamma_1)=\{u_0=v_0,u_1=v_1\}$ and $E(\Gamma_1)=\{(u_0,u_1)\}$. For $r\ge 2$, let $\Gamma_r$ be the directed graph with
\begin{align*}
  V(\Gamma_r)&=\{u_0=v_0,u_1,v_1,u_2,v_2,\ldots,u_{r-1},v_{r-1},u_r=v_r\}\\
  E(\Gamma_r)&=\{(u_i,u_{i+1}),(v_i,v_{i+1})\mid i=0,1,\ldots,r-1\}.
\end{align*}
Namely, for $r\ge 2$, $\Gamma_r$ is depicted as follows.

\begin{figure}[H]
  \label{Gamma_n}
  \centering
  \begin{tikzpicture}[x=5mm, y=5mm, thick]
    \draw[myarrow=.6](-0.2,0)--(1,2);
    \draw[myarrow=.6](1,2)--(3,2);
    \draw(3,2)--(4,2);
    \draw[dashed](4,2)--(6,2);
    \draw(6,2)--(7,2);
    \draw[myarrow=.6](7,2)--(9,2);
    \draw[myarrow=.6](9,2)--(10.2,0);
    \draw[myarrow=.6](-0.2,0)--(1,-2);
    \draw[myarrow=.6](1,-2)--(3,-2);
    \draw(3,-2)--(4,-2);
    \draw[dashed](4,-2)--(6,-2);
    \draw(6,-2)--(7,-2);
    \draw[myarrow=.6](7,-2)--(9,-2);
    \draw[myarrow=.6](9,-2)--(10.2,0);
    \fill[black](-0.2,0) node[left]{$u_0=v_0$} circle(2pt);
    \fill[black](1,2) node[above]{$u_1$} circle(2pt);
    \fill[black](3,2) node[above]{$u_2$} circle(2pt);
    \fill[black](7,2) node[above]{$u_{r-2}$} circle(2pt);
    \fill[black](9,2) node[above]{$u_{r-1}$} circle(2pt);
    \fill[black](10.2,0) node[right]{$u_r=v_r$} circle(2pt);
    \fill[black](1,-2) node[below]{$v_1$} circle(2pt);
    \fill[black](3,-2) node[below]{$v_2$} circle(2pt);
    \fill[black](7,-2) node[below]{$v_{r-2}$} circle(2pt);
    \fill[black](9,-2) node[below]{$v_{r-1}$} circle(2pt);
  \end{tikzpicture}
  \caption{The directed graph $\Gamma_r$ for $r\ge 2$}
\end{figure}

\noindent We say that a reduced path in $\Gamma_r$ ($2\le r<\infty$) is clockwise if it is so with respect to the above picture. Let $x$ be a vertex of $\Gamma_r$. For $r=0$, let $\rho_x\colon I\to\Gamma_0$ denote the constant map for $I\in\mathcal{I}_1$, where $\mathcal{I}_1$ consists of two elements $\vec{I}_1$ and the directed graph obtained by reversing the edge of $\vec{I}_1$. For $r=1$, let $\rho_x\colon I_2\to\Gamma_1$ with $I_2\in\I_2$ be the unique reduced path with $\rho_x(0)=x=\rho_x(2)$. For $r\ge 2$, let $\rho_x\colon I_{2r}\to\Gamma_r$ with $I_{2r}\in\I_{2r}$ be the unique clockwise reduced path with $\rho_x(0)=x=\rho_x(2r)$.

\begin{definition}
  \label{n-homotopy def}
  Let $f$ and $g$ be paths in a directed graph $X$. For $0\le r<\infty$, we say that $f$ is \emph{$C_r$-homotopic} to $g$, denoted $f\to_rg$, if either of the following conditions holds.

  \begin{enumerate}
    \item $f=g$.

    \item $f=f_1\cdot f_2$ and $g=f_1\cdot(h\circ\rho_x)\cdot f_2$ for some paths $f_1,f_2$ in $X$ and a map $h\colon \Gamma_s\to X$ with $s=0$ or $r$, where $x$ is a vertex of $\Gamma_r$ such that $h(x)$ is the end point of $f_1$.
  \end{enumerate}
\end{definition}

Note that if paths $f$ and $g$ in a directed graph are $C_r$-homotopic, then their initial and terminal points are respectively common. If there is a finite zig-zag of $C_r$-homotopies between paths $f$ and $g$, then we write $f\approx_rg$ and also say that $f$ and $g$ are $C_r$-homotopic. We show the basic properties of $C_r$-homotopies.

\begin{lemma}
  \label{C_r-homotopy}
  Let $f_1,f_2,g_1,g_2$ be paths in a directed graph $X$, and let $\varphi\colon X\to Y$ be a map of directed graphs. Then for $0\le r<\infty$, the following hold.

  \begin{enumerate}
    \item The relation $\approx_r$ is an equivalence relation.

    \item If $f_1\approx_rf_2$ and $g_1\approx_rg_2$, then $f_1\cdot g_1\approx_rf_2\cdot g_2$.

    \item If $f_1\approx_rg_1$, then $\varphi\circ f_1\approx_r\varphi\circ g_1$.
  \end{enumerate}
\end{lemma}

\begin{proof}
  Immediate.
\end{proof}

Let $I_n\in\I_n$. We define $\bar{I}_n\in\I_n$ by reversing all edges of $I_n$. For a path $f\colon I_n\to X$, we define its inverse path $\bar{f}\colon\bar{I}_n\to X$ by $\bar{f}(i)=f(n-i)$ for $i=0,1,\ldots,n$.

\begin{lemma}
  \label{inverse}
  For any path $f\colon I_n\to X$ from $x$ to $y$ with $I_n\in\I_n$, we have
  \[
    f\cdot\bar{f}\approx_1c_x\quad\text{and}\quad\bar{f}\cdot f\approx_1c_y.
  \]
\end{lemma}

\begin{proof}
  It is sufficient to prove the first equivalence because the proof for the second equivalence is quite similar. We prove the first equivalence by induction on $n$. For $n=0$, the statement is trivial, so we may assume $n\ge 1$. We consider the case $(f(n-1),f(n))\in E(X)$. Define a map $h\colon\Gamma_1\to X$ by $h(u_0)=f(n-1)$ and $h(u_1)=f(n)$. Then
  \begin{equation}
    \label{C_r contraction}
    f\cdot\bar{f}=f_{[n-1]}\cdot(h\circ\rho_{u_0})\cdot\overline{f_{[n-1]}}
  \end{equation}
  where $f_{[n-1]}=f\vert_{(I_n)_{[n-1]}}$, and so $f\cdot\bar{f}\to_1f_{[n-1]}\cdot\overline{f_{[n-1]}}$. Thus the induction proceeds. The case that $(f(n),f(n-1))\in E(X)$ is similarly proved. If $f(n-1)=f(n)$, then \eqref{C_r contraction} holds, where $h$ is the constant map to the vertex $f(n)$. Thus $f\cdot\bar{f}\to_1f_{[n-1]}\cdot\overline{f_{[n-1]}}$ too, so the induction proceeds. Therefore the statement follows.
\end{proof}

\begin{lemma}
  \label{m<n}
  Let $f$ and $g$ be paths in a directed graph $X$. If $1\le r\le s<\infty$, then $f\approx_rg$ implies $f\approx_sg$.
\end{lemma}

\begin{proof}
  We set
  \[
    \alpha(u_i)=
    \begin{cases}
      u_i&0\le i\le r-1\\
      u_r&r\le i\le s
    \end{cases}
    \quad\text{and}\quad
    \alpha(v_i)=
    \begin{cases}
      v_i&0\le i\le r-1\\
      v_r&r\le i\le s.
    \end{cases}
  \]
  Obviously, $\alpha$ is a surjective map from $\Gamma_s$ onto $\Gamma_r$. Hence, given a map $h\colon\Gamma_s\to X$ and a vertex $x$ of $\Gamma_s$, $h\circ\rho_x$ and $h\circ\alpha\circ\rho_{\alpha(x)}$ are $C_0$-homotopic, and thus the statement follows.
\end{proof}

Let $f$ and $g$ be paths in a directed graph $X$. Recall from \cite{GLMY2} that $f$ is $C$-homotopic to $g$ if there are paths $f_1,f_2,h$ in $X$ such that $f=f_1\cdot f_2$ and $g=f_1\cdot h\cdot f_2$, where $h$ is the constant map or the inclusion of either of the following loops.

\begin{figure}[htbp]
  \centering
  \begin{tikzpicture}[x=5mm, y=5mm, thick]
    \draw[myarrow=.6](-3,1) to [out=140,in=40] (-6,1);
    \draw[myarrow=.6](-6,1) to [out=-40,in=-140] (-3,1);
    \fill[black](-3,1) circle(2pt);
    \fill[black](-6,1) circle(2pt);
    \draw[myarrow=.6](0,0)--(3,0);
    \draw[myarrow=.6](1.5,2)--(0,0);
    \draw[myarrow=.6](1.5,2)--(3,0);
    \fill[black](1.5,2) circle(2pt);
    \fill[black](0,0) circle(2pt);
    \fill[black](3,0) circle(2pt);
    \draw[myarrow=.6](6,0)--(9,0);
    \draw[myarrow=.6](9,0)--(9,2);
    \draw[myarrow=.6](6,0)--(6,2);
    \draw[myarrow=.6](6,2)--(9,2);
    \fill[black](6,0) circle(2pt);
    \fill[black](9,0) circle(2pt);
    \fill[black](9,2) circle(2pt);
    \fill[black](6,2) circle(2pt);
  \end{tikzpicture}
  \caption{$C$-homotopies}
\end{figure}

\noindent Observe that if two paths in a directed graph with common end points have the same image, then they are $C_1$-homotopic, hence $C_2$-homotopic by Lemma \ref{m<n}. Clearly, a single point and the above loops are exactly the images of maps from $\Gamma_2$. Then we get:

\begin{lemma}
  \label{C-homotopy}
  Two paths in a directed graph are $C$-homotopic if and only if they are $C_2$-homotopic.
\end{lemma}


\subsection{Combinatorial description}

We define the groupoid $\widetilde{\Pi}_1^r(X)$ of a directed graph $X$ combinatorially by using $C_r$-homotopies for $1\le r<\infty$.

\begin{definition}
  \label{groupoid def}
  For $1\le r<\infty$, we define the groupoid $\widetilde{\Pi}_1^r(X)$ for a directed graph $X$ as the groupoid whose objects are vertices of $X$ and the hom-set $\widetilde{\Pi}_1^{r}(X)(x,y)$ is the set of $C_r$-homotopy equivalence classes of paths from $x$ to $y$.
\end{definition}

Indeed, the groupoid $\widetilde{\Pi}_1^r(X)$ is well-defined by Lemmas \ref{concatenation}, \ref{C_r-homotopy}, \ref{inverse} and \ref{m<n}. Remark that as well as the usual fundamental groupoid, we employ the notation $f\cdot g$ for the composite $g\circ f$ of morphisms $f\colon x\to y$ and $g\colon y\to z$ in $\widetilde{\Pi}_1^r(X)$.

Now we prove:

\begin{theorem}
  \label{equivalence def}
  For $1\le r<\infty$, there is a natural isomorphism
  \[
    \widetilde{\Pi}_1^{r}(X)\cong\mathcal{G}^{r}(X).
  \]
\end{theorem}

For composable morphisms $f,g$ of $\mathcal{G}^{r}(X)$, we set
\[
  f\cdot g=g\circ f
\]
as well as $\Pi_1^{r}(X)$. Now we define a functor
\[
  F^{r}\colon\widetilde{\Pi}_1^{r}(X)\to\mathcal{G}^{r}(X)
\]
by $F^r([f])=(f(0),f(1))$ and $F^r([\bar{f}])=(f(1),f(0))^{-1}$ for a reduced directed path $f\colon\vec{I}_1\to X$, that is, an edge.

\begin{lemma}
  The functor $F^{r}\colon\widetilde{\Pi}_1^{r}(X)\to\mathcal{G}^r(X)$ is well-defined.
\end{lemma}

\begin{proof}
  Observe that every path is $0$-homotopic to a reduced path and every reduced path is a sequence of finitely many edges of $X$. Then the functor $F^{r}$ is well-defined whenever $F^{r}([f])=F^{r}([g])$ for paths $f,g$ in $X$ such that $f\to_rg$, that is, $f=g_1\cdot g_2$ and $g=g_1\cdot(h\circ\rho_x)\cdot g_2$ for some paths $g_1,g_2$ in $X$ and a map $h\colon\Gamma_r\to X$. Let $e_i$ and $f_i$ denote edges $(h(u_i),h(u_{i+1})$ and $(h(v_i),h(v_{i+1}))$ of $X$, respectively. If $x=u_i$, then
  \[
    F^r([h\circ\rho_x])=e_i\cdots e_{r-1}\cdot\overline{f_{r-1}}\cdots\overline{f_0}\cdot e_0\cdots e_{i-1}
  \]
  and so $F^{r}([h\circ\rho_x])=1_{h(x)}$. If $x=v_i$, then we can see $F^r([h\circ\rho_x])=1_{h(x)}$ quite similarly. Thus we obtain $F^{r}([f])=F^{r}([g])$.
\end{proof}

We also define a functor
\[
  G^{r}\colon\mathcal{G}^r(X)\to\widetilde{\Pi}_1^{r}(X)
\]
by $G^{r}(e)=[e]$ for $e\in E(X)$ and $G^{r}((x,x))=1_x$ for $x\in V(X)$.

\begin{lemma}
  The functor $G^{r}\colon\widetilde{\Pi}_1^{r}(X)\to\Pi_1^{r}(X)$ is well-defined.
\end{lemma}

\begin{proof}
  Let $e_1,\ldots,e_p,f_1,\ldots,f_q\in\mathcal{G}^r(X)$ for $1\le p,q\le r$ such that $e_1\cdots e_p$ and $f_1\cdots f_q$ are defined. It is sufficient to show that $e_1\cdots e_p\approx_rf_1\cdots f_q$ whenever $s(e_1)=s(f_1)$ and $t(e_p)=t(f_q)$. Observe that $e_1\cdots e_p$ and $f_1\cdots f_q$ define a map $\Gamma_r\to X$ whenever $s(e_1)=s(f_1)$ and $t(e_p)=t(f_q)$. Then we get $e_1\cdots e_p\approx_rf_1\cdots f_q$, completing the proof.
\end{proof}

\begin{proof}
  [Proof of Theorem \ref{equivalence def}]
  By definition, $F^{r}$ and $G^{r}$ are mutually inverse, and natural with respect to a directed graph $X$. Thus the statement follows.
\end{proof}

By Proposition \ref{algebric description} and Theorem \ref{equivalence def}, there is a natural isomorphism
\[
  \Pi_1^r(X)\cong\widetilde{\Pi}_1^r(X).
\]
Then in the sequel, we will not distinguish these two groupoids. By Lemma \ref{C-homotopy}, we can recover:

\begin{corollary}
  [Di, Ivanov, Mukoseev, and Zhang {\cite[Theorem 2.4]{DIMZ}}]
  \label{2-fundamental groupoid}
  The $2$-fundamental groupoid of a directed graph is naturally isomorphic to the fundamental groupoid defined by Grigor'yan, Jimenez, and Muranov \cite{GJM}.
\end{corollary}

We give a combinatorial description of $\Pi_1^\infty(X)$. For a directed graph $X$, let $\widetilde{\Pi}_1^\infty(X)$ be the groupoid whose objects are vertices of $X$ and morphisms from $x$ to $y$ is the eqivalence class of paths $f$ and $g$ from $x$ to $y$, where two paths are equivalent if $f\approx_rg$ for some $1\le r<\infty$.

\begin{corollary}
  For a directed graph $X$, there is a natural isomorphism
  \[
    \Pi_1^\infty(X)\cong\widetilde{\Pi}_1^\infty(X).
  \]
\end{corollary}

\begin{proof}
  By Proposition \ref{colim} and the definition of $\widetilde{\Pi}_1^\infty(X)$, there are isomorphisms
  \[
    \Pi_1^\infty(X)\cong\underset{r\to\infty}{\mathrm{colim}}\,\widetilde{\Pi}_1^r(X)\cong\widetilde{\Pi}_1^\infty(X).
  \]
\end{proof}


\subsection{Homotopy invariance}

We recall from \cite{HR1} the definition of an $r$-homotopy between maps of directed graphs for $1\le r<\infty$, and also define an $\infty$-homotopy.

\begin{definition}
  Let $\varphi_0,\varphi_1\colon X\to Y$ be maps of directed graphs.
  \begin{enumerate}
    \item For $1\le r<\infty$, an \emph{$r$-homotopy} from $\varphi_0$ to $\varphi_1$ is a family $h=\{h_x\}_{x\in V(X)}$ where $h_x\colon\vec{I}_r\to X$ is a directed path from $\varphi_0(x)$ to $\varphi_1(x)$.

    \item An \emph{$\infty$-homotopy} from $\varphi_0$ to $\varphi_1$ is a family $h=\{h_x\}_{x\in V(X)}$ where $h_x$ is a directed path from $\varphi_0(x)$ to $\varphi_1(x)$.
  \end{enumerate}
\end{definition}

Remark that a $1$-homotopy is a homotopy considered by Grigor'yan, Lin, Muranov, and Yau \cite{GLMY1,GLMY2}, and an $\infty$-homotopy is a long homotopy considered by Hepworth and Roff \cite{HR1}. We use the name $\infty$-homotopy, instead of a long homotopy, because we prefer to consider it and $r$-homotopies for $1\le r<\infty$ simultaneously.

Let $\varphi_0,\varphi_1\colon X\to Y$ be maps of directed graphs. If there is a zig-zag of $r$-homotopies between $\varphi_0$ and $\varphi_1$, then we write $\varphi_0\simeq_r\varphi_1$ and say that $\varphi_0$ and $\varphi_1$ are $r$-homotopic. Analogously to Lemma \ref{m<n}, we have:

\begin{lemma}
  \label{m<n homotopy}
  Let $\varphi_0,\varphi_1\colon X\to Y$ be maps of directed graphs. If $\varphi_0\simeq_r\varphi_1$, then for any $r\le s\le\infty$, $\varphi_0\simeq_s\varphi_1$.
\end{lemma}

We say that directed graphs $X$ and $Y$ are \emph{$r$-homotopy equivalent} if there are maps $\phi\colon X\to Y$ and $\psi\colon Y\to X$ such that $\phi\circ\psi\simeq_r1_Y$ and $\psi\circ\phi\simeq_r1_X$. By Lemma \ref{m<n homotopy}, if $X$ and $Y$ are $r$-homotopy equivalent, then $X$ and $Y$ are $s$-homotopy equivalent for any $s\ge r$. We prove the homotopy invariance of $r$-fundamental groupoids.

\begin{proposition}
  \label{natural transformation}
  Let $\varphi_0,\varphi_1\colon X\to Y$ be maps of directed graphs, and let $\Pi_1^{s}(\varphi_0),$\\$\Pi_1^{s}(\varphi_1)\colon\Pi_1^s(X)\to\Pi_1^s(Y)$ be the associated functors. Then an $r$-homotopy $h$ from $\varphi_0$ to $\varphi_1$ defines a natural isomorphism of the functors
  \[
    \alpha(h)\colon\Pi_1^{s}(\varphi_0)\to\Pi_1^{s}(\varphi_1)
  \]
  for $r+1\le s\le\infty$.
\end{proposition}

\begin{proof}
  Let $h=\{h_x\}_{x\in V(X)}$ be an $r$-homotopy from $\varphi_0$ to $\varphi_1$, where $h_x$ is a directed path $h_x\colon\vec{I}_r\to X$ from $\varphi_0(x)$ to $\varphi_1(x)$. We set $\alpha(h)_x=[h_x]\in\Pi_1^{s}(X)(\varphi_0(x),\varphi_1(x))$ for $x\in V(X)$ and $r+1\le s\le\infty$. Observe that for an edge $e=(u,v)\in E(X)$, the path $(\varphi_0\circ e)\cdot h_v\cdot\overline{(\varphi_1\circ e)}\cdot\overline{h_u}$ defines a $C_{r+1}$-homotopy from $(\varphi_0\circ e)\cdot h_v$ to $h_u\cdot(\varphi_1\circ e)$, where we regard the edge $e$ as a directed path. Then by Lemma \ref{m<n}, the diagram
  \[
    \xymatrix{
      \varphi_0(u)\ar[r]^{[\varphi_0\circ e]}\ar[d]_{\alpha(h)_u}&\varphi_0(v)\ar[d]^{\alpha(h)_v}\\
      \varphi_1(u)\ar[r]^{[\varphi_1\circ e]}&\varphi_1(v)
    }
  \]
  commutes in $\Pi_1^{s}(X)$, and so we get a natural transformation $\alpha(h)\colon\Pi_1^{s}(\varphi_0)\to\Pi_1^{s}(\varphi_1)$. Let $\bar{h}=\{\overline{h_x}\}_{x\in V(X)}$. Then $\bar{h}$ is an $r$-homotopy from $\varphi_1$ to $\varphi_0$, and by definition, $\alpha(h)$ and $\alpha(\bar{h})$ are mutually inverse. Thus the statement follows.
\end{proof}

By Proposition \ref{natural transformation}, we can immediately see the homotopy invariance of $r$-fundamental groupoids.

\begin{corollary}
  \label{homotopy invariance}
  If directed graphs $X$ and $Y$ are $r$-homotopy equivalent, then for $r+1\le s\le\infty$, there is a categorical equivalence
  \[
    \Pi_1^{s}(X)\simeq\Pi_1^{s}(Y).
  \]
\end{corollary}


\subsection{Product formula}

We prove the product formula for $r$-fundamental groupoids. Since the fundamental groupoid of a directed graph defined by Grigor'yan, Jimenez, and Muranov \cite{GJM} satisfies the products formula, it follows from Corollary \ref{2-fundamental groupoid} that the $2$-fundamental groupoid does so. For directed graphs $X$ and $Y$, let $p_1\colon X\square Y\to X$ and $p_2\colon X\square Y\to Y$ denote projections.

\begin{lemma}
  \label{square}
  For any $I\in\I_k$ and $J\in\I_l$, all paths from $(0,0)$ to $(k,l)$ in $I\square J$ are $C_2$-homotopic.
\end{lemma}

\begin{proof}
  We prove the statement by induction on $k+l$. Let $f$ and $g$ be paths from $(0,0)$ to $(k,l)$ in $I\square J$. If $k=0$ or $l=0$, then the statement is trivially true. Hence the $k+l=0,1$ case is done, and we may assume $k,l\ge 1$. Let $g$ be the reduced path in $I\square J$ obtained by concatenating $I\square 0$ and $k\square J$. Clearly, every path is $0$-homotopic to a reduced path, so we consider a reduced path $f$ from $(0,0)$ to $(k,l)$ in $I\square J$. Then $f(1)=(1,0)$ or $(0,1)$. If $f(1)=(1,0)$, then the induction proceeds by replacing $I$ with a shorter directed graph $I_{\{1,2,\ldots,m\}}$. Suppose that $f(1)=(0,1)$. Since $k,l\ge 1$, there is $I\square J$ has the subgraph $I_{\{0,1\}}\square J_{\{0,1\}}=\vec{I}_1\square\vec{I}_1=\Gamma_2$. Then $f$ is $C_2$-homotopic to a path $\tilde{f}$ from $(0,0)$ to $(k,l)$ in $I\square J$ with $\tilde{f}(1)=(1,0)$, so the induction proceeds.
\end{proof}

\begin{lemma}
  \label{commutative}
  Let $f\colon I\to X$ and $g\colon J\to Y$ be paths in directed graphs $X$ and $Y$, where $I\in\I_k$ and $J\in\I_l$. Then
  \[
    (f,g(0))\cdot(f(k),g)\approx_2(f(0),g)\cdot(f,g(l)).
  \]
\end{lemma}

\begin{proof}
  We may consider $(f,g(0))\cdot(f(k),g)$ is a map from the subgraph $(I\square 0)\cup(k\square J)$ of $I\square J$, where $I\square 0$ is sent by $(f,g(0))$ and $k\square J$ is sent by $(f(k),g)$. Similarly, we may consider $(f(0),g)\cdot(f,g(l))$ as a map from the subgraph $(0\square J)\cup(I\square l)$ of $I\square J$. Then by Lemma \ref{square}, the statement follows.
\end{proof}

\begin{theorem}
  \label{product formula}
  Let $X$ and $Y$ be directed graphs. Then the map
  \[
    (p_1)_*\times(p_2)_*\colon\Pi_1^{r}(X\square Y)\to\Pi_1^{r}(X)\times\Pi_1^{r}(Y)
  \]
  is an isomorphism for $2\le r\le\infty$.
\end{theorem}

\begin{proof}
  We first consider the $2\le r<\infty$ case. Clearly, the map $(p_1)_*\times(p_2)_*$ is the identity map on objects and surjective on hom-sets. Then it remains to show that $(p_1)_*\times(p_2)_*$ is injective. Observe that every path in $X\square Y$ is the composite of finitely many paths of the form $(f,y)$ and $(x,g)$, where $f,g$ are paths in $X,Y$ and $x,y$ are vertices of $X,Y$. Then by Lemma \ref{commutative}, any path in $X\square Y$ is $C_2$-homotopic to a path $(f,g(0))\cdot(f(k),g)$ for some paths $f\colon I\to X$ and $g\colon J\to Y$ with $I\in\I_k$ and $J\in\I_l$. Observe that $(p_1)_*\times(p_2)_*([(f,g(0))\cdot(f(k),g)])=([f],[g])$. Thus $(p_1)_*\times(p_2)_*$ is injective. The $r=\infty$ case follows from the $2\le r<\infty$ case, Proposition \ref{colim} and the fact that finite limits commute with filtered colimits.
\end{proof}

\begin{remark}
  Theorem \ref{product formula} does not generally hold for $r=1$. For example, by Proposition \ref{1-fundamental groupoid}, $\Pi_1^{1}(\vec{I}_1\square\vec{I}_1)$ has nontrivial automorphisms, but $\Pi_1^{1}(\vec{I}_1)\times\Pi_1^{1}(\vec{I}_1)$ has trivial automorphisms only.
\end{remark}

It is easy to see that the proof of Theorem \ref{product formula} works verbatim for the strong product $X\boxtimes Y$, and so we get:

\begin{corollary}
  Let $X$ and $Y$ be directed graphs. Then the map
  \[
    (p_1)_*\times(p_2)_*\colon\Pi_1^{r}(X\boxtimes Y)\to\Pi_1^{r}(X)\times\Pi_1^{r}(Y)
  \]
  is an isomorphism for $2\le r\le\infty$, where $p_i$ denotes the $i$-th projection for $i=1,2$.
\end{corollary}


\section{Hurewicz theorem}\label{Hurewicz theorem}

In this section, we recall the magnitude-path spectral sequence, and prove Theorem \ref{main 1}. We also consider a relation between the directed graph $\Gamma_r$ and a differential in the magnitude-path spectral sequence.


\subsection{Magnitude-path spectral sequence}

Let $X$ be a directed graph. Let $\mathrm{RC}_*(X)$ be the normalized Moore complex of $\mathrm{N}(X)$. Namely, $\mathrm{RC}_n(X)$ is the free abelian group generated by $(x_0,\ldots,x_n)\in V(X)^{n+1}$ such that $L(x_0,\ldots,x_n)<\infty$ and $x_i\ne x_{i+1}$ for $i=0,1\ldots,n-1$, and the boundary map is given by
\[
  \partial(x_0,\ldots,x_n)=\sum_{i=0}^n(-1)^n(x_0,\ldots,\widehat{x_i},\ldots,x_n)
\]
where we set $(x_0,\ldots,\widehat{x_i},\ldots,x_n)=0$ if $x_{i-1}=x_{i+1}$.

\begin{definition}
  [Hepworth and Roff \cite{HR2}]
  The \emph{reachability homology} of a directed graph $X$ is defined by
  \[
    \mathrm{RH}_*(X)=H_*(\mathrm{RC}_*(X)).
  \]
\end{definition}

Let $F_p\mathrm{RC}_*(X)$ be the normalized Moore complex of $\mathrm{N}^p(X)$. Then $F_p\mathrm{RC}_*(X)$ is a subcomplex of $\mathrm{RC}_*(X)$ such that $F_p\mathrm{RC}_n(X)$ is the free abelian group generated by $(x_0,\ldots,x_n)\in\mathrm{RC}_n(X)$ with $L(x_0,\ldots,x_n)\le p$. Now we get the exhausting filtration
\begin{equation}
  \label{filtration}
  F_0\mathrm{RC}_*(X)\subset F_1\mathrm{RC}_*(X)\subset F_2\mathrm{RC}_*(X)\subset\cdots\subset \mathrm{RC}_*(X).
\end{equation}

\begin{definition}
  [Asao \cite{A}]
  The \emph{magnitude-path spectral sequence} (MPSS, for short) $E_{p,q}^r(X)$ of a directed graph $X$ is the spectral sequence associated to the filtration \eqref{filtration}.
\end{definition}

We employ the standard notation for homology spectral sequences though Asao \cite{A} employs a nonstandard one. Since $F_p\mathrm{RC}_{p+q}(X)=0$ for $p<0$ or $q>0$, the MPSS is a fourth quadrant spectral sequence such that the differential $d^r$ has degree $(-r,r-1)$. Let $\mathrm{MH}_*^*(X)$ and $\mathrm{PH}_*(X)$ denote the magnitude homology and the path homology of a directed graph $X$. As mentioned in Section \ref{Introduction}, we have:

\begin{proposition}
  [Asao {\cite[Proposition 7.8]{A}}]
  For a directed graph $X$, there are isomorphisms
  \[
    E_{p,q}^1(X)\cong\mathrm{MH}_{p+q}^p(X)\quad\text{and}\quad E_{p,0}^2(X)\cong\mathrm{PH}_p(X).
  \]
\end{proposition}


\subsection{Proof of Theorem \ref{main 1}}

Recall that a pointed directed graph is a pair $(X,x_0)$ of a directed graph $X$ and a distinguished vertex $x_0$ called the basepoint. We can define maps between pointed directed graphs and $r$-homotopies between maps of pointed directed graphs in the obvious way. In particular, we can consider $r$-homotopy equivalences of pointed directed graphs.

\begin{definition}
  For $1\le r\le\infty$, the $r$-fundamental group of a pointed directed graph $(X,x_0)$ is defined by
  \[
    \pi_1^{r}(X,x_0)=\Pi_1^{r}(X)(x_0,x_0).
  \]
\end{definition}

Remark that by Proposition \ref{1-fundamental groupoid}, $\pi_1^{1}(X,x_0)$ is isomorphic to the edge-path group of the underlying undirected graph $\widehat{X}$ with basepoint $x_0$. Recall that the fundamental group of a pointed directed graph $(X,x_0)$ defined by  Grigor'yan, Lin, Muranov, and Yau \cite{GLMY2} is denoted by $\pi_1^\mathrm{GLMY}(X,x_0)$. By Lemma \ref{C-homotopy}, we get:

\begin{corollary}
  \label{2-fundamental group}
  For a pointed directed graph $(X,x_0)$, there is a natural isomorphism
  \[
    \pi_1^2(X,x_0)\cong\pi_1^\mathrm{GLMY}(X,x_0).
  \]
\end{corollary}

By Corollary \ref{homotopy invariance}, we have:

\begin{corollary}
  If pointed directed graphs $(X,x_0)$ and $(Y,y_0)$ are $r$-homotopy equivalent, then for any $r+1\le s\le\infty$, there is an isomorphism
  \[
    \pi_1^{s}(X,x_0)\cong\pi_1^{s}(Y,y_0).
  \]
\end{corollary}

The sequence \eqref{sequence Pi} of $r$-fundamental groupoids restricts to the sequence of $r$-fundamental groups
\begin{equation}
  \label{sequence pi}
  \pi_1^{1}(X,x_0)\to\pi_1^{2}(X,x_0)\to\pi_1^{3}(X,x_0)\to\cdots
\end{equation}
where all maps are surjective.

\begin{proof}
  [Proof of Theorem \ref{main 1}]
  For $2\le r\le\infty$, there is the Hurewicz map
  \[
    \pi_1(|\mathrm{N}^r(X)|,x_0)\to H_1(|\mathrm{N}^r(X)|)
  \]
  which is identified with the abelianization. By Lemma \ref{realization}, there is a natural isomorphism $\pi_1(|\mathrm{N}^r(X)|,x_0)\cong\pi_1^r(X,x_0)$. Clearly, there is a natural isomorphism $H_1(|\mathrm{N}^r(X)|)\cong H_1(\mathrm{N}^r(X))=H_1(F_r\mathrm{RC}_*(X))$ too, where we set $F_\infty\mathrm{RC}_*(X)=\mathrm{RC}_*(X)$. Then the above Hurewicz map induces a natural map
  \[
    h^r\colon\pi_1^r(X,x_0)\to H_1(F_r\mathrm{RC}_*(X))
  \]
  which is identified with the abelianization. Hence, the $r=\infty$ case is proved. By the naturality of the map $h^r$, there is a commutative diagram
  \[
    \xymatrix{
      \pi_1^r(X,x_0)\ar[r]\ar[d]_{h^r}&\pi_1^{r+1}(X,x_0)\ar[d]^{h^{r+1}}\\
      H_1(F_r\mathrm{RC}_*(X))\ar[r]&H_1(F_{r+1}\mathrm{RC}_*(X))
    }
  \]
  where the horizontal maps are the natural maps. Then as the top and the right maps are surjective, so is the bottom map too. In particular, the map
  \[
    H_1(F_1\mathrm{RC}_*(X))\to H_1(F_r\mathrm{RC}_*(X))
  \]
  is surjective. On the other hand, by \cite[Chapter XV, Section 1, (8)]{CE}, there is a natural isomorphism
  \[
    E^r_{1,0}\cong\mathrm{Im}\{H_1(F_1\mathrm{RC}_*(X))\to H_1(F_r\mathrm{RC}_*(X)/F_0\mathrm{RC}_*(X))\}
  \]
  for $r\ge 2$. Since $F_0\mathrm{RC}_1(X)=0$, we get $H_1(F_0\mathrm{RC}_*(X))=0$, implying that the map
  \[
    H_1(F_r\mathrm{RC}_*(X))\to H_1(F_r\mathrm{RC}_*(X)/F_0\mathrm{RC}_*(X))
  \]
  is injective. Thus we obtain a natural isomorphism
  \[
    E^r_{1,0}\cong H_1(F_r\mathrm{RC}_*(X))
  \]
  and therefore the $r<\infty$ case is proved, completing the proof.
\end{proof}


\subsection{Hurewicz map}

We consider a combinatorial description of the Hurewicz map $h^r$. Let $X$ be a directed graph, and let $f\colon I\to X$ be a path for $I\in\I_n$. We define an element $h(f)$ of $F_1\mathrm{RC}_1(X)$ by
\[
  h(f)=\sum_{(i,i+1)\in E(I)}f(i)f(i+1)-\sum_{(i+1,i)\in E(I)}f(i+1)f(i).
\]

\begin{lemma}
  \label{hur}
  Let $f$ and $g$ be paths in a directed graph $X$.

  \begin{enumerate}
    \item $h(f\cdot g)=h(f)+h(g)$.

    \item $h(\bar{f})=-h(f)$.

    \item If $f$ is a loop, then $\partial(h(f))=0$.
  \end{enumerate}
\end{lemma}

\begin{proof}
  (1) and (2) are obvious. Observe that if $f$ is a map from some element of $\I_n$, then
  \begin{align*}
    \partial(h(f))&=\sum_{(i,i+1)\in E(X)}\partial(f(i)f(i+1))-\sum_{(i+1,i)\in E(X)}\partial(f(i+1)f(i))\\
    &=\sum_{(i,i+1)\in E(X)}(f(i+1)-f(i))-\sum_{(i+1,i)\in E(X)}(f(i)-f(i+1))\\
    &=\sum_{i=0}^{n-1}(f(i+1)-f(i))\\
    &=f(n)-f(0).
  \end{align*}
  Thus $\partial(h(f))=0$ whenever $f$ is a loop, proving (3).
\end{proof}

\begin{lemma}
  \label{theta}
  Let $X$ be a directed graph. Then for a map $f\colon\Gamma_r\to X$ with $r\ge 1$ and any vertex $x$ of $\Gamma_r$, we have
  \[
    h(f\circ\rho_x)=\sum_{i=0}^{r-1}(f(u_i)f(u_{i+1})-f(v_i)f(v_{i+1})).
  \]
\end{lemma}

\begin{proof}
  If $x=u_k$, then
  \begin{align*}
    h(f\circ\rho_x)&=\sum_{i=k}^{r-1}f(u_i)f(u_{i+1})-\sum_{i=0}^{r-1}f(v_{r-i-1})f(v_{n-i})+\sum_{i=0}^{k-1}f(u_i)f(u_{i+1})\\
    &=\sum_{i=0}^{r-1}(f(u_i)f(u_{i+1})-f(v_i)f(v_{i+1})).
  \end{align*}
  The $x=v_k$ case can be similarly proved.
\end{proof}

Let $Z$ be the kernel of the boundary map $\partial\colon F_1\mathrm{RC}_1(X)\to F_1\mathrm{RC}_0(X)$. Then for $r\ge 2$,
\begin{equation}
  \label{Z-E}
  E^r_{1,0}(X)=Z/Z\cap\partial F_r\mathrm{RC}_2(X).
\end{equation}
In particlar, there is the canonical projection
\[
  Z\to E^1_{1,0}(X)
\]
By Lemma \ref{hur}, the map $h\colon\pi_1^{1}(X,x_0)\to Z$ is a well-defined homomorphism. For $r\ge 2$, we define
\[
  \tilde{h}^r\colon\pi_1^{r}(X,x_0)\to E^r_{1,0}(X),\quad[f]\mapsto[\tilde{h}(f)].
\]

\begin{proposition}
  For $r\ge 2$, the map $\tilde{h}^r$ is a well-defined homomorphism.
\end{proposition}

\begin{proof}
  Let $r\ge 2$. Let $f$ and $g$ be loops in $X$ based at $x_0$. Suppose that $f\to_rg$. Let $f,g$ be paths in $X$ satisfying $f\to_rg$. Then there is a map $k\colon\Gamma_r\to X$ such that $f=f_1\cdot f_2$ and $g=f_1\cdot(k\circ\rho_x)\cdot f_2$ for some paths $f_1,f_2$ in $X$ and some vertex $x$ of $\Gamma_r$. Hence by Lemmas \ref{hur} and \ref{theta},
  \begin{align*}
    h(g)-h(f)&=(h(f_1)+h(k\circ\rho_x)+h(f_2))-(h(f_1)+h(f_2))\\
    &=h(k\circ\rho_x)\\
    &=\sum_{i=0}^{r-1}(k(u_i)k(u_{i+1})-k(v_i)k(v_{i+1})).
  \end{align*}
  We set
  \[
    \sigma=\sum_{i=0}^{r-2}(k(u_i)k(u_{i+1})k(u_r)-k(v_i)k(v_{i+1})k(v_r))
  \]
  which is an element of $F_r\mathrm{RC}_2(X)$. Then
  \begin{align*}
    \partial\sigma&=\sum_{i=0}^{r-2}(k(u_{i+1})k(u_r)-k(u_i)h(u_r)+k(u_i)k(u_{i+1})\\
    &\quad-k(v_{i+1})k(v_r)+k(v_i)k(v_r)-k(v_i)k(v_{i+1}))\\
    &=\sum_{i=0}^{r-1}(k(u_i)k(u_{i+1})-k(v_i)k(v_{i+1}))\\
    &=h(g)-h(f).
  \end{align*}
  Thus by \eqref{Z-E}, the map $h^r$ is well-defined. By Lemma \ref{hur}, the map $h^r$ is a homomorphism, completing the proof.
\end{proof}

\begin{proposition}
  \label{Hurewicz map}
  For any connected pointed directed graph, $\tilde{h}^r$ coincides with $h^r$.
\end{proposition}

\begin{proof}
  By definition, $h$ coincides with the Hurewicz map
  \[
    \pi_1^1(X,x_0)=\pi_1(\mathrm{N}_1(X),x_0)\to H_1(\mathrm{N}_1(X))=Z.
  \]
  There is a commutative diagram
  \[
    \xymatrix{
      \pi_1^1(X,x_0)\ar[r]\ar[d]^h&\pi_1^2(X,x_0)\ar[r]\ar[d]^{\tilde{h}^2}&\pi_1^3(X,x_0)\ar[r]\ar[d]^{\tilde{h}^3}&\cdots\\
      Z\ar[r]&E^2_{1,0}(X)\ar[r]&E^3_{1,0}(X)\ar[r]&\cdots
    }
  \]
  where the horizontal maps are the canonical ones. Since the horizontal maps are surjections, the maps $\tilde{h}^i$ for $i\ge 2$ are determined by the map $h$. Thus the statement follows.
\end{proof}

\begin{remark}
  The proof of Proposition \ref{Hurewicz map} implies that the diagram \ref{ladder h} extends to the left by the commutative diagram
  \[
  \xymatrix{
    \pi_1^1(X,x_0)\ar[r]\ar[d]_{h}&\pi_1^2(X,x_0)\ar[d]^{h^2}\\
    Z\ar[r]&E^2_{1,0}(X).
  }
  \]
\end{remark}


\subsection{Differential in the MPSS}

We consider how the directed graph $\Gamma_r$ is related with a differential in the MPSS. Let $r\ge 2$, and let $X$ be a directed graph. It is straightforward to see that $E^0_{p,q}(\Gamma_r)$ is the free abelian group generated by
\[
  (u_{i_0},\ldots,u_{i_{p+q}})\quad\text{and}\quad(v_{i_0},\ldots,v_{i_{p+q}})
\]
where $0\le i_0<\cdots<i_{p+q}\le r$ and $i_{p+q}-i_0=p$, and the differential is given by $d^0u_i=d^0v_i=0$ and
\begin{align*}
  d^0(u_{i_0},\ldots,u_{i_{p+q}})&=\sum_{k=1}^{p+q-1}(-1)^k(u_{i_0},\ldots,\widehat{u_{i_k}},\ldots,u_{i_{p+q}})\\
  d^0(v_{i_0},\ldots,v_{i_{p+q}})&=\sum_{k=1}^{p+q-1}(-1)^k(v_{i_0},\ldots,\widehat{v_{i_k}},\ldots,v_{i_{p+q}}).
\end{align*}
Then it is straightforward to see that
\[
  E_{p,q}^1(\Gamma_r)=
  \begin{cases}
    \Z\{u_0,\ldots,u_r,v_0,\ldots,v_r\}&(p,q)=(0,0)\\
    \Z\{(u_i,u_{i+1}),(v_i,v_{i+1})\mid i=0,1,\ldots,r-1\}&(p,q)=(1,0)\\
    \Z\{[(u_0,u_1,u_r)-(v_0,v_1,v_r)]\}&(p,q)=(r,2-r)\\
    0&\text{otherwise}
  \end{cases}
\]
where $\Z S$ denotes the free abelian group generated by a set $S$. Observe that
\[
  d^1(u_i,u_{i+1})=u_{i+1}-u_i\quad\text{and}\quad d^1(v_i,v_{i+1})=v_{i+1}-v_i.
\]
Thus we get
\[
  E_{p,q}^2(\Gamma_r)\cong
  \begin{cases}
    \Z&(p,q)=(1,0),(r,2-r)\\
    0&\text{otherwise}.
  \end{cases}
\]
By degree reasons, $E^2(\Gamma_r)=\cdots=E^{r-1}(\Gamma_r)$, and so a possible first nontrivial diffrential is $d^{r-1}$. By Proposition \ref{1-fundamental groupoid}, $\pi_1^1(\Gamma_r,u_0)$ is a free abelian group generated by a loop $\rho_{u_0}\colon I\to\Gamma_r$ for $I\in\I_{2n}$. Then $\pi_1^r(\Gamma_r,u_0)=0$ as $\rho_{u_0}\approx_rc_{u_0}$, and so by Theorem \ref{main 1}, $E^r_{1,0}(\Gamma_r)=0$. Thus we obtain:

\begin{proposition}
  \label{differential}
  For $2\le r<\infty$, the differential
  \[
    d^{r-1}\colon E^{r-1}_{r,2-r}(\Gamma_r)\to E^{r-1}_{1,0}(\Gamma_r)
  \]
  is an isomorphism, where $E^{r-1}_{r,2-r}(\Gamma_r)\cong E^{r-1}_{1,0}(\Gamma_r)\cong\Z$.
\end{proposition}

\begin{remark}
  In fact, we can prove the following stronger result. Let $X$ be a directed graph, and let $2\le r<\infty$. For any $x\in E_{r,2-r}^{r-1}(X)$, there are maps $f_1,\ldots,f_n\colon\Gamma_r\to X$ such that $x$ is a linear combination of $(f_1)_*(y),\ldots,(f_ns)_*(y)$, where $n$ depends on $x$ and $y$ is a generator of $E^{r-1}_{r,2-r}(\Gamma_r)\cong\Z$. The proof is a complicated and long, so we omit it here.
\end{remark}


\section{Seifert-van Kampen theorem}

In this section, we define the $r$-separability of a pair of directed graphs, and prove Theorem \ref{main 2}. As its application, we also prove the Seifert-van Kampen theorem for the pushout along an $r$-cofibration of directed graphs, which is a relaxation of a cofibration defined by Carranza \textit{et al.} \cite{CD}.


\subsection{$r$-Separability}

We begin with an example showing that without assuming any condition, the Seifert-van Kampen theorem for the $r$-fundamental groupoid does not hold.

\begin{example}
  Let $X$ and $Y$ be the following directed graphs.

  \begin{figure}[H]
    \centering
    \begin{tikzpicture}[x=5mm, y=5mm, thick]
      \draw[myarrow=.6](0,0)--(2,0);
      \draw(2,0)--(3,0);
      \draw[dashed](3,0)--(5,0);
      \draw(5,0)--(6,0);
      \draw[myarrow=.6](6,0)--(8,0);
      \draw[myarrow=.6](0,0)--(4,3);
      \draw[myarrow=.6](4,3)--(8,0);
      \fill[black](0,0) node[below]{$x_0$} circle(2pt);
      \fill[black](2,0) node[below]{$x_1$} circle(2pt);
      \fill[black](6,0) node[below]{$x_r$} circle(2pt);
      \fill[black](8,0) node[below]{$x_{r+1}$} circle(2pt);
      \fill[black](4,3) node[above]{$a$} circle(2pt);
      \draw[myarrow=.6](12,0)--(14,0);
      \draw(14,0)--(15,0);
      \draw[dashed](15,0)--(17,0);
      \draw(17,0)--(18,0);
      \draw[myarrow=.6](18,0)--(20,0);
      \draw[myarrow=.6](12,0)--(16,3);
      \draw[myarrow=.6](16,3)--(20,0);
      \fill[black](12,0) node[below]{$x_0$} circle(2pt);
      \fill[black](14,0) node[below]{$x_1$} circle(2pt);
      \fill[black](18,0) node[below]{$x_r$} circle(2pt);
      \fill[black](20,0) node[below]{$x_{r+1}$} circle(2pt);
      \fill[black](16,3) node[above]{$b$} circle(2pt);
    \end{tikzpicture}
    \caption{The directed graphs $X$ and $Y$}
  \end{figure}
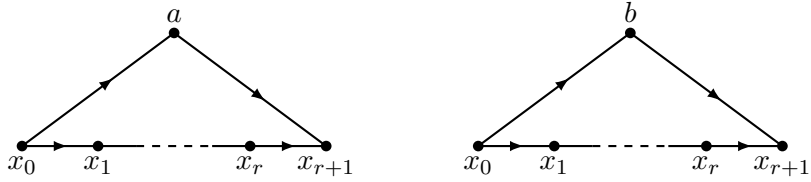

  \noindent Then we have $\pi_1^{r}(X,x_0)$ and $\pi_1^{r}(Y,x_0)$ are free abelian groups generated by the clockwise loops $\alpha_X$ and $\alpha_Y$, respectively. Since $X\cap Y$ is $1$-homotopy equivalent to a point, we have $\pi_1^{r}(X\cap Y,x_0)=0$. On the other hand, $\alpha_X\cdot\overline{\alpha_Y}$ is $C_1$-homotopic to a $C_2$-homotopy given by $\Gamma_2$ with vertices $x_0,x_{r+1},a,b$ in $X\cup Y$. Then for $r\ge 2$, $[\alpha_X][\alpha_Y]^{-1}=1$ in $\pi_1^{r}(X\cup Y,x_0)$, implying that the commutative diagram
  \[
    \xymatrix{
      \Pi_1^{r}(X\cap Y)\ar[r]\ar[d]&\Pi_1^{r}(X)\ar[d]\\
      \Pi_1^{r}(Y)\ar[r]&\Pi_1^{r}(X\cup Y)
    }
  \]
  is not a pushout.
\end{example}

Let $X$ be a directed graph, and let $A$ be a subgraph of $X$. We define the \emph{complement} $X-A$ to be the smallest subgraph of $X$ satisfying $X=(X-A)\cup A$. More explicitly, $X-A$ is a subgraph of $X$ defined by $E(X-A)=E(X)-E(A)$ and
\[
  V(X-A)=(V(X)-V(A))\cup s(E(X-A))\cup t(E(X-A))
\]
where $s,t\colon E(X)\to V(X)$ are maps given by $s(x,y)=x$ and $t(x,y)=y$.

\begin{example}
  Let $X$ be a directed graph. If $f\colon I\to X$ is a reduced loop and $g\colon J\to X$ is a reduced path from $f(k)$ to $f(l)$ for some $0\le k,l\le m$, where  $I\in\I_m$ and $J\in\I_n$. Then
  \[
    f(I)\cup g(J)-g(J)=f(I)
  \]
\end{example}

Let $f\colon I\to X$ be a reduced path in a directed graph $X$ with $I\in\I_n$. A \emph{subpath} of $f$ is the restriction of $f$ to the induced subgraph $I_{\{k,k+1,\ldots,l\}}$ for some $0\le k\le l\le n$. We say that a subgraph of a directed graph is represented by a path if it is the image of a path.

\begin{definition}
  Let $X$ be a directed graph, and let $L$ be a subgraph of $X$ represented by a reduced loop. A \emph{subdivision} of $L$ consists of two subgraphs $L_1$ and $L_2$ of $X$ represented by reduced loops such that $L_1\cap L_2$ is represented by subpaths of $L_1,L_2$ and
  \[
    L=L_1\cup L_2-L_1\cap L_2.
  \]
\end{definition}

    %

We can iteratedly subdivide a subgraph $L$ represented by a reduced loop, and the resulting subgraphs $L_1,\ldots,L_n$ represented by reduced loops are also called a subdivision of $L$. A subgraph of a directed graph $X$ is called a \emph{degenerate $\Gamma_r$} if it is obtained from $\Gamma_r$ by contracting some edges and identifying some vertices. Namely, a degenerate $\Gamma_r$ in $X$ is the image of a map $\Gamma_r\to X$. Then every $C_r$-homotopy in $X$ is given by a degenerate $\Gamma_r$ in $X$.

\begin{definition}
  Let $(X,Y)$ be a pair of directed graphs
  \begin{enumerate}
    \item For $1\le r<\infty$, $(X,Y)$ is \emph{$r$-separable} if any degenerate $\Gamma_r$ in $X\cup Y$ is subdivided into degenerate $\Gamma_r$'s, each of which is in either $X$ or $Y$.

    \item $(X,Y)$ is \emph{$\infty$-separable} if for any $1\le r<\infty$, any degenerate $\Gamma_r$ in $X\cup Y$ is subdivided into degenerate $\Gamma_s$'s for some $s\ge r$ possibly depending on the given degenerate $\Gamma_r$, each of which is in either $X$ or $Y$.
  \end{enumerate}
\end{definition}

Remark that any pair of directed graph is $1$-separable. We record a trivial example of an $r$-separable pair, for which Grigor'yan, Jimenez, and Muranov \cite{GJM} prove the Seifert-van Kampen theorem of the fundamental groupoids (= $2$-fundamental groupoids).

\begin{lemma}
  \label{trivially separable}
  Let $X$ and $Y$ be directed graphs. For $1\le r<\infty$, if every degenerate $\Gamma_r$ in $X\cup Y$ is in either $X$ or $Y$, then the pair $(X,Y)$ is $r$-separable.
\end{lemma}

Note that if $(X,Y)$ satisfies the condition of Lemma \ref{trivially separable} for all $1\le r<\infty$, then it is $\infty$-separable.

\begin{example}
  We give an example of an $r$-separable pair of directed graphs that does not satisfy the condition of Lemma \ref{trivially separable}. For $r\ge 2$, let us consider the following directed graph.

  \begin{figure}[H]
    \centering
    \begin{tikzpicture}[x=5mm, y=5mm, thick]
      \draw[myarrow=.6](-0.2,0)--(1,2);
      \draw[myarrow=.6](1,2)--(3,2);
      \draw(3,2)--(4,2);
      \draw[dashed](4,2)--(6,2);
      \draw(6,2)--(7,2);
      \draw[myarrow=.6](7,2)--(9,2);
      \draw[myarrow=.6](9,2)--(10.2,0);
      \draw[myarrow=.6](-0.2,0)--(1,-2);
      \draw[myarrow=.6](1,-2)--(3,-2);
      \draw(3,-2)--(4,-2);
      \draw[dashed](4,-2)--(6,-2);
      \draw(6,-2)--(7,-2);
      \draw[myarrow=.6](7,-2)--(9,-2);
      \draw[myarrow=.6](9,-2)--(10.2,0);
      \draw[myarrow=.6](9,2)--(9,-2);
      \fill[black](-0.2,0) node[left]{$u_0=v_0$} circle(2pt);
      \fill[black](1,2) node[above]{$u_1$} circle(2pt);
      \fill[black](3,2) node[above]{$u_2$} circle(2pt);
      \fill[black](7,2) node[above]{$u_{r-2}$} circle(2pt);
      \fill[black](9,2) node[above]{$u_{r-1}$} circle(2pt);
      \fill[black](10.2,0) node[right]{$u_r=v_r$} circle(2pt);
      \fill[black](1,-2) node[below]{$v_1$} circle(2pt);
      \fill[black](3,-2) node[below]{$v_2$} circle(2pt);
      \fill[black](7,-2) node[below]{$v_{r-2}$} circle(2pt);
      \fill[black](9,-2) node[below]{$v_{r-1}$} circle(2pt);
    \end{tikzpicture}
  \end{figure}

  \noindent Let $X$ and $Y$ be the induced subgraphs over $\{u_0,\ldots,u_{r-1},v_1,\ldots,v_{r-1}\}$ and $\{u_{r-1},u_r,v_{y-1}\}$, respectively. Then it is obvious that $(X,Y)$ is an $r$-separable pair, but not satisfies the condition of Lemma \ref{trivially separable} as the perimeter $\Gamma_r$ in $X\cup Y$ is in neither $X$ nor $Y$.
\end{example}

We apply $r$-separability to decompose $C_r$-homotopies as follows.

\begin{lemma}
  \label{subdivision composite}
  Let $X$ be a directed graph. Suppose that reduced loops $f\colon I\to X$ and $g\colon J\to X$ for $I\in\I_m$ and $J\in\I_n$ satisfy $f(k)=g(0)$ and $f(I)\cap g(J)$ is represented by a subpath $f\vert_{I_{\{k,k+1,\ldots,k+l\}}}$ for some $0\le k\le n$ and $0\le l\le m-k$. Then
  \begin{equation}
    \label{composite}
    (f\vert_{I_{[k]}})\cdot\overline{(g\vert_{J_{\{l,l+1,\ldots,n\}}})}\cdot(f\vert_{I_{\{k+l,k+l+1,\ldots,m\}}})\approx_1(f\vert_{I_{[k]}})\cdot\bar{g}\cdot(f\vert_{I_{\{k,k+1,\ldots,m\}}}).
  \end{equation}
\end{lemma}

\begin{proof}
  By assumption, $f(I)\cap g(J)$ is also represented by $g\vert_{J_{[l]}}$, and so
  \begin{align*}
    &(f\vert_{I_{[k]}})\cdot\bar{g}\cdot(f\vert_{I_{\{k,k+1,\ldots,m\}}})\\
    &=(f\vert_{I_{[k]}})\cdot\overline{(g\vert_{J_{\{l,l+1,\ldots,n\}}})}\cdot\overline{(g\vert_{J_{[l]}})}\cdot(f\vert_{I_{\{k,k+1,\ldots,k+l\}}})\cdot(f\vert_{I_{\{k+l,k+l+1,\ldots,m\}}})\\
    &\approx_1(f\vert_{I_{[k]}})\cdot\overline{(g\vert_{J_{\{l,l+1,\ldots,n\}}})}\cdot(f\vert_{I_{\{k+l,k+l+1,\ldots,m\}}}).
  \end{align*}
  Then the statement is proved.
\end{proof}

Observe that the LHS of \eqref{composite} represents $f(I)\cup g(J)-f(I)\cap g(J)$, and that the RHS of \eqref{composite} represents $f(I)\cup g(J)$. We call the RHS of \eqref{composite} the \emph{composite} of $f_1$ and $f_2$.

\begin{proposition}
  \label{subdivision}
  Suppose that $(X,Y)$ is an $r$-separable pair of directed graphs.
  \begin{enumerate}
    \item If $1\le r<\infty$, then every $C_r$-homotopy in $X\cup Y$ can be replaced by the composite of some $C_r$-homotopies in either $X$ or $Y$.
    \item If $r=\infty$, then every $C_s$-homotopy with any $1\le s<\infty$ can be replaced by the composite of some $C_t$-homotopies in either $X$ or $Y$ for some $t\ge s$.
  \end{enumerate}
\end{proposition}

\begin{proof}
  By definition, every $C_1$-homotopy in $X\cup Y$ is either in $X$ or $Y$. Observe that if a subgraph $L$ of a directed graph is represented by some reduced loop, then for any vertex $x$ of $L$, there is a reduced loop based at $x$ which represents $L$. Then the statement follows from Lemma \ref{subdivision composite}.
\end{proof}


\subsection{Proof of Theorem \ref{main 2}}

Let $X$ and $Y$ be directed graphs, and let $1\le r\le\infty$ throughout this subsection. We define the groupoid $\mathcal{G}^{r}$ by the pushout
\[
  \xymatrix{
    \Pi_1^{r}(X\cap Y)\ar[r]\ar[d]&\Pi_1^{r}(X)\ar[d]\\
    \Pi_1^{r}(Y)\ar[r]&\mathcal{G}^{r}.
    }
\]
Then objects of $\mathcal{G}^{r}$ are vertices of $X\cup Y$, and the hom-set $\mathcal{G}^{r}(x,y)$ consists of the formal composite
\[
  x\xrightarrow{f_1}x_1\xrightarrow{f_2}\cdots\xrightarrow{f_n}x_n=y
\]
where each $f_i$ is a morphism of either $\Pi_1^{r}(X)$ or $\Pi_1^{r}(Y)$. Let
\[
  F^{r}\colon\mathcal{G}^{r}\to\Pi^{r}_1(X\cup Y).
\]
denote the natural functor. We aim to prove that the functor $F^{r}$ is an isomorphism whenever $(X,Y)$ is $r$-separable.

\begin{proposition}
  \label{F^1}
  The functor $F^{1}$ is an isomorphism.
\end{proposition}

\begin{proof}
  Let $f\colon I\to X\cup Y$ be a path from $x$ to $y$, where $I\in\I_n$, and let $f_i$ denote the restriction of $f$ to the induced subgraph $I_{\{i,i+1\}}$ for $i=0,1,\ldots,n-1$. Then $f$ is identified with a vertex or an edge of $X\cup Y$, so that it lies in either $X$ or $Y$. Now we define a functor $G^{1}\colon\Pi^{1}_1(X\cup Y)\to\mathcal{G}^{1}$ by the identity map on objects and
  \[
    G^{1}([f])=[f_0]\cdot[f_1]\cdots[f_{n-1}].
  \]
  Let us show that $G^1$ is well-defined. Let $g\colon J\to X\cup Y$ be a path from $x$ to $y$ such that $f\to_1g$. Then there is a map $h\colon\Gamma_1\to X\cup Y$ such that
  \[
    g=f_0\cdots f_{k-1}\cdot(h\circ\rho_z)\cdot f_k\cdots f_{n-1}
  \]
  for some $0\le k\le n-1$ and a vertex $z$ of $\Gamma_1$. Observe that $h\circ\rho_z=e\cdot\bar{e}$ for some path $e\colon I_1\to X\cup Y$ with $I_1\in\I_1$. Note that $e(I_1)$ is  contained in either $X$ or $Y$. Then
  \begin{align*}
    G^{1}([g])&=[f_0]\cdots[f_{k-1}]\cdot[e]\cdot[e]^{-1}\cdot[f_k]\cdots[f_{n-1}]\\
    &=[f_0]\cdots[f_{k-1}]\cdot[f_k]\cdots[f_{n-1}]\\
    &=G^{1}([f]).
  \end{align*}
  Thus $G^{1}$ is well-defined.

  For $i=1,2,\ldots,k$, let $f_i\colon I_i\to X\cup Y$ be paths for $I_i\in\I_{n_i}$ such that each $f_i$ is in either $X$ or $Y$. Then
  \[
    F^{1}([f_1]\cdots[f_k])=[f_1]\cdots[f_k].
  \]
  Let $f_i^j$ denote the restriction of $f_i$ to the induced subgraph $(I_i)_{\{j,j+1\}}$ for $i=1,2,\ldots,k$ and $j=0,1,\ldots,n_i-1$. Then $f_i^j$ is identified with a vertex or an edge of $X\cup Y$. By definition, we have
  \[
    G^{1}([f_1]\cdots[f_k])=[f_1^0]\cdots[f_1^{n_1-1}]\cdots[f_k^0]\cdots[f_k^{n_k-1}]=[f_1]\cdots[f_k].
  \]
  Thus $F^{1}$ and $G^{1}$ are mutually inverse.
\end{proof}

Remark that Proposition \ref{F^1} also follows from Proposition \ref{1-fundamental groupoid} and the Seifert-van Kampen theorem for spaces. Let
\[
  P^{r}\colon\Pi_1^{1}(X)\to\Pi_1^{r}(X)\quad\text{and}\quad Q^{r}\colon\mathcal{G}^{1}\to\mathcal{G}^{r}
\]
denote the composite of the natural functors in the sequence \eqref{sequence Pi} and the functor induced from $P^{r}$, respectively. Then there is a commutative diagram
\begin{equation}
  \label{Q diagram1}
  \xymatrix{
    \mathcal{G}^{1}\ar[d]_{Q^{r}}\ar[r]^(.36){F^{1}}&\Pi_1^{1}(X\cup Y)\ar[d]^{P^{r}}\\
    \mathcal{G}^{r}\ar[r]^(.36){F^{r}}&\Pi_1^{r}(X\cup Y).
  }
\end{equation}

\begin{lemma}
  \label{contraction}
  If a map $f\colon\Gamma_r\to X\cup Y$ factors through $X$ or $Y$, then
  \[
    Q^{r}\circ G^{1}(f\circ \rho_x)=1
  \]
  for any vertex $x$ of $\Gamma_r$.
\end{lemma}

\begin{proof}
  Quite similarly to the proof of Lemma \ref{F^1}, we can show that
  \[
    Q^{r}\circ G^{1}(f\circ \rho_x)=[f\circ\rho_x].
  \]
  Then since $[f\circ\rho_x]=[c_x]=1$ in $\mathcal{G}^{r}(f(x),f(x))$, the statement follows.
\end{proof}

\begin{lemma}
  \label{n-separable}
  Let $(X,Y)$ be an $r$-separable pair of directed graphs. Then there is a functor $G^{r}\colon\Pi_1^{r}(X\cup Y)\to\mathcal{G}^{r}$ satisfying the commutative diagram
  \begin{equation}
    \label{Q diagram2}
    \xymatrix{
      \Pi_1^{1}(X\cup Y)\ar[r]^(.65){G^{1}}\ar[d]_{P^{r}}&\mathcal{G}^{1}\ar[d]^{Q^{r}}\\
      \Pi_1^{r}(X\cup Y)\ar[r]^(.65){G^{r}}&\mathcal{G}^{r}.
    }
  \end{equation}
\end{lemma}

\begin{proof}
  We put $H=Q^{r}\circ G^{1}$. Let $f,g\colon I\to X\cup Y$ be paths from $x$ to $y$. It is sufficient to prove $H([f])=H([g])$ whenever $f\to_rg$. Suppose that $f\to_rg$, that is, there are paths $f_1,f_2$ in $X\cup Y$ and a map $h\colon\Gamma_r\to X\cup Y$ such that $f=f_1\cdot f_2$ and $g=f_1\cdot(h\circ\rho_x)\cdot f_2$ for some vertex $x$ of $\Gamma_r$. Then
  \[
    H(f)=H(f_1)\cdot H(f_2),\quad H(g)=H(f_1)\cdot H(h\circ\rho_x)\cdot H(f_2).
  \]
  Observe that each $C_0$-homotopy and $C_1$-homotopy in $X\cup Y$ are in either $X$ or $Y$. Then since $(X,Y)$ is $r$-separable, we can apply Proposition \ref{subdivision} to the map $H$, so that by Lemma \ref{contraction}, we obtain $H([h\circ\rho_x])=1$.
\end{proof}

Now we are ready to prove Theorem \ref{main 2}.

\begin{proof}
  [Proof of Theorem \ref{main 2}]
  By juxtaposing the commutative diagrams \eqref{Q diagram1} and \eqref{Q diagram2}, we can see that $F^{r}$ and $G^{r}$ are mutually inverse.
\end{proof}

We consider corollaries of Theorem \ref{main 2}.

\begin{corollary}
  \label{van Kampen pi_1}
  Let $(X,x_0)$ and $(Y,x_0)$ be pointed directed graphs with common basepoint. If $(X,Y)$ is $r$-separable and $X\cap Y$ is connected, then the commutative diagram
  \[
    \xymatrix{
      \pi_1^{r}(X\cap Y,x_0)\ar[r]\ar[d]&\pi_1^{r}(X,x_0)\ar[d]\\
      \pi_1^{r}(Y,x_0)\ar[r]&\pi_1^{r}(X\cup Y,x_0)
    }
  \]
  is a pushout.
\end{corollary}

\begin{proof}
  We may assume that $X$ and $Y$ are connected. If a directed graph $Z$ is connected, then the inclusion functor $\pi_1^{r}(Z,z_0)\to\Pi_1^{r}(Z)$ has a left inverse. Then the statement follows from Theorem \ref{main 2}.
\end{proof}

\begin{proof}
  [Proof of Corollary \ref{Mayer-vietoris sequence}]
  By Theorem \ref{main 2} and Corollary \ref{van Kampen pi_1}, the commutative diagram
  \[
    \xymatrix{
      E^r_{1,0}(X\cap Y)\ar[r]^{(i_X)_*}\ar[d]_{(i_Y)_*}&E^r_{1,0}(X)\ar[d]^{(j_X)_*}\\
      E^r_{1,0}(Y)\ar[r]^{(j_Y)_*}&E^r_{1,0}(X\cup Y)
    }
  \]
  is a pushout, and then the statement follows.
\end{proof}


\subsection{$r$-Cofibration}

For the rest of this section, we consider the pushout of directed graphs
\begin{equation}
  \label{pushout phi1}
  \xymatrix{
    A\ar[r]^{\varphi_X}\ar[d]_{\varphi_Y}&X\ar[d]\\
    Y\ar[r]&X\cup_AY.
  }
\end{equation}
In \cite{CD}, Carranza \textit{et al.} proved that the category of directed graphs carries a cofibration category structure where weak equivalences are maps inducing isomorphisms in path homology. As mentioned in Section \ref{Introduction}, Hepworth and Roff \cite{HR1} proved excision holds for the MPSS of $X\cup_AY$ whenever $\varphi_X$ is a cofibration, and they also proved the Mayer-Vietoris sequence of the MPSS of $X\cup_AY$ for $r=1,2$, and asked whether there is the Mayer-Vietoris sequence for $r\ge 3$. Hepworth and Roff \cite{HR2} also  introduced a long cofibration by relaxing the definition of a cofibration, and prove that reachability homology enjoys excision and the Mayer-Vietoris sequence for the pushout \eqref{pushout phi1} whenever $\varphi_X$ is a long cofibration. Here, we introduce an $r$-cofibration by relaxing the definition of a cofibration, and as an application of Theorem \ref{main 2}, prove the Seifert-van Kampen theorem for the pushout \eqref{pushout phi1} whenever $\varphi_X$ is an $r$-cofibration. As a consequence, we get the Mayer-Vietoris sequence of the $E^r_{1,0}$ of the MPSS for $X\cup_AY$ whenever $\varphi_X$ is an $r$-cofibration, which is a partial answer to the question of Hepworth and Roff's question \cite{HR1} mentioned in Section \ref{Introduction}.

We define an $r$-cofibration. Let $A$ be a subgraph of a directed graph $X$. The \emph{reach} of $A$, denoted $\mathrm{r}(A)$, is defined to be the set of all vertices of $X$ having a directed path from some vertex of $A$. The reach of $A$ is defined as an induced subgraph of $X$ by Hepworth and Roff \cite{HR1}, but we choose this definition because we do not need a directed graph structure.

\begin{definition}
  \label{def cofibration}
  For $1\le r<\infty$, an \emph{$r$-cofibration} of directed graphs is an induced subgraph inclusion $A\to X$ such that
  \begin{enumerate}
    \item $d(x,a)=\infty$ for any $x\in V(X)-V(A)$ and $a\in V(A)$, and

    \item there is a map $\pi\colon\mathrm{r}(A)\to V(X)$ satisfying that for any $a\in V(A)$ and $x\in\mathrm{r}(A)$, $d(a,\pi(x))\le d(a,x)$ and
    \[
      d(a,\pi(x))+d(\pi(x),x)
      \begin{cases}
      \le r&d(a,x)\le r\\
      =d(a,x)&d(a,x)>r.
      \end{cases}
    \]
  \end{enumerate}
\end{definition}

Remark that a cofibration of Carranza \textit{et al.} \cite{CD} is exactly a $1$-cofibration. We also define an $\infty$-cofibration which is also called a long cofibration by Hepworth and Roff \cite{HR1}.

\begin{definition}
  An \emph{$\infty$-cofibration} of directed graphs is an induced subgraph inclusion $A\to X$ such that
  \begin{enumerate}
    \item $d(x,a)=\infty$ for any $x\in V(X)-V(A)$ and $a\in V(A)$, and

    \item there is a map $\pi\colon\mathrm{r}(A)\to V(X)$ satisfying that for any $a\in V(A)$ and $x\in\mathrm{r}(A)$, $d(a,x)<\infty$ if and only if $d(a,\pi(x))<\infty$.
  \end{enumerate}
\end{definition}

\begin{lemma}
  \label{r-cofibration}
  For $1\le r\le s\le\infty$, an $r$-cofibration is an $s$-cofibration.
\end{lemma}

\begin{proof}
  The only nontrivial case is $r<\infty$ and $s=\infty$. Let $A\to X$ be an $r$-cofibration for $r<\infty$, ane let $a\in V(A)$ and $x\in\mathrm{r}(A)$. For proving that $A\to X$ is an $\infty$-cofibration, it is sufficient to show that $d(a,\pi(x))<\infty$ implies $d(a,x)<\infty$. Since $x\in\mathrm{r}(A)$, there is $b\in V(A)$ such that $d(b,x)<\infty$, implying that $d(b,\pi(x))+d(\pi(x),x)\le\max\{r,d(b,x)\}$. Then $d(\pi(x),x)<\infty$, and so
  \[
    d(a,x)\le d(a,\pi(x))+d(\pi(x),x)<\infty
  \]
  whenever $d(a,\pi(x))<\infty$. Thus the proof is finished.
\end{proof}

We define a directed graph $\widetilde{Y}$ by
\begin{align*}
  V(\widetilde{Y})&=V(Y)\sqcup V(A)\\
  E(\widetilde{Y})&=E(Y)\sqcup E(A)\sqcup\{(a,\varphi_Y(a))\mid a\in V(A)\}.
\end{align*}
Then $\widetilde{Y}$ is a mapping cylinder of $\varphi_Y$ defined by Grigor'yan, Lin, Muranov, and Yau \cite{GLMY2}, so that $A$ and $Y$ are subgraphs of $\widetilde{Y}$ and there is a natural map $\rho\colon\widetilde{Y}\to Y$ such that $\rho\vert_Y=1_Y$ and $\rho(a)=\varphi_Y(a)$ for $a\in V(A)\subset V(\widetilde{Y})$.



\begin{lemma}
  \label{(X,Y)}
  If the first condition in Definition \ref{def cofibration} is satisfied, then the pair $(X,\widetilde{Y})$ is $r$-separable for $1\le r\le\infty$.
\end{lemma}

\begin{proof}
  By the definition of $\widetilde{Y}$, there is no directed path from $V(Y)$ to $V(X)$. By assumption, there is no directed path from $V(X)-V(A)$ to $V(Y)$ either. Then any map $\Gamma_s\to X\cup_A\widetilde{Y}$ with $1\le s<\infty$ factors through $X$ or $\widetilde{Y}$, and so by Lemma \ref{trivially separable}, $(X,\widetilde{Y})$ is $r$-separable.
\end{proof}

Now we consider the pushout of directed graphs
\[
  \xymatrix{
    A\ar[r]^{\varphi_X}\ar[d]&X\ar[d]\\
    \widetilde{Y}\ar[r]&X\cup_A\widetilde{Y}
  }
\]
where the left map is the inclusion. Through the map $\rho\colon\widetilde{Y}\to Y$, there is a natural map $\tilde{\rho}\colon X\cup_A\widetilde{Y}\to X\cup_AY$. By Theorem \ref{main 2} and Lemma \ref{(X,Y)}, we get:

\begin{proposition}
  \label{pushout lift}
  For $1\le r\le\infty$, the commutative diagram
  \[
    \xymatrix{
      \Pi_1^{r}(A)\ar[r]\ar[d]&\Pi_1^{r}(X)\ar[d]\\
      \Pi_1^{r}(\widetilde{Y})\ar[r]&\Pi_1^{r}(X\cup_A\widetilde{Y})
    }
  \]
  is a pushout.
\end{proposition}

For $k=0,1,\ldots,n$, let $\vec{I}_{n+1}(k)$ denote an element of $\I_{n+1}$ given by
\[
  E(\vec{I}_n(k))=\{(0,1),\ldots,(k-1,k),(k+1,k),(k+1,k+2),\ldots,(n,n+1)\}.
\]
Namely, $\vec{I}_{n+1}(k)$ is obtained from $\vec{I}_{n+1}$ by reversing the edge $(k,k+1)$.

\begin{lemma}
  \label{lift}
  Let $f\colon\vec{I}_n\to X\cup_AY$ be a path from $V(Y)$ to $V(X)-V(A)$. If $\varphi_X$ is an $r$-cofibration, then for each $a\in V(A)$ with $\varphi_Y(a)=f(k)$, there is a path $\tilde{f}_a\colon\vec{I}_{n+1}(k)\to X\cup_A\widetilde{Y}$ for some $0\le k\le n$ such that $\tilde{f}_a(k+1)=a\in V(A)$ and
  \begin{equation}
    \label{lift f}
    \tilde{f}_a(i)=
    \begin{cases}
      f(i)\in V(Y)&i=0,1,\ldots,k\\
      f(i-1)\in V(X)-V(A)&i=k+2,\ldots,n+1.
    \end{cases}
  \end{equation}
  Moreover, $\tilde{f}_a\approx_r\tilde{f}_b$ for any $b\in V(A)$ with $\varphi_Y(b)=f(k)$.
\end{lemma}

\begin{proof}
  We first consider the $1\le r<\infty$ case. Since $\varphi_X$ is an $r$-cofibration, there is no directed path from $V(X)-V(A)$ to $V(Y)$ in $X\cup_AY$. Then there is a unique $0\le k\le r$ such that $f(i)\in V(Y)$ for $0\le i\le k$ and $f(i)\in V(X)-V(A)$ for $k+1\le i\le r$. In this case, we must have $f(k)\in V(\varphi_Y(A))$. Now we take any $a\in V(A)$ with $\varphi_Y(a)=f(k)$ and $(a,x)\in E(X)$. Then we can define $\tilde{f}_a\colon\vec{I}_{r+1}(k)\to X\cup_A\widetilde{Y}$ by \eqref{lift f} and $\tilde{f}(k+1)=a$. Thus the first statement follows. Let $x=f(k+1)$ and $y=f(k)$. Since $\varphi_X$ is an $r$-cofibration and $d(a,x)=1\le r$, we have
  \[
    d(a,\pi(x))+d(\pi(x),x)\le r.
  \]
  Since $x\in V(X)-V(A)$, we have $d(\pi(x),x)\ge 1$, implying $d(a,\pi(x))\le r-1$. Then we get
  \[
    d(a,\pi(x))+d(\pi(x),\varphi_Y(\pi(x)))\le r\quad\text{and}\quad d(a,y)+d(y,\varphi_Y(\pi(x)))\le r.
  \]
  Then we obtain the following subgraph of $X\cup_A\widetilde{Y}$ such that the upper path from $x$ to $y$ through $a$ is $\tilde{f}_a$, and the left square and the right triangle are degenerate $\Gamma_r$'s, where the arrow from $\pi(x)$ to $x$ is a directed path and the remaining arrows are edges.

  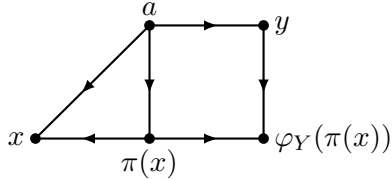
\begin{figure}[H]
    \centering
    \begin{tikzpicture}[x=5mm, y=5mm, thick]
      \draw[myarrow=.6](3,3)--(6,3);
      \draw[myarrow=.6](6,3)--(6,0);
      \draw[myarrow=.6](3,3)--(0,0);
      \draw[myarrow=.6](3,0)--(6,0);
      \draw[myarrow=.6](3,3)--(3,0);
      \draw[myarrow=.6](3,0)--(0,0);
      \fill[black](6,3) node[right]{$y$} circle(2pt);
      \fill[black](0,0) node[left]{$x$} circle(2pt);
      \fill[black](3,0) node[below]{$\pi(x)$} circle(2pt);
      \fill[black](3,3) node[above]{$a$} circle(2pt);
      \fill[black](6,0) node[right]{$\varphi_Y(\pi(x))$} circle(2pt);
    \end{tikzpicture}
    \caption{Deformation of $\tilde{f}_a$}
  \end{figure}

  \noindent Therefore by Lemma \ref{subdivision} and the observation after it, $\tilde{f}_a$ is $C_r$-homotopic to the bottom-right perimeter path from $x$ to $y$ in the above figure, which is independent of the choice of $a$. Thus the second statement follows.

  Forgetting about the length of paths, the above proof works for $r=\infty$, completing the proof.
\end{proof}

\begin{lemma}
  \label{lift homotopy}
  Let $f,g\colon\vec{I}_r\to X\cup_AY$ be paths from $V(Y)$ to $V(X)-V(A)$. If $\varphi_X$ is an $r$-cofibration and $f\approx_rg$, then $\tilde{f}_a\approx_r\tilde{g}_b$, where $\tilde{f}_a$ and $\tilde{g}_b$ are as in Lemma \ref{lift}.
\end{lemma}

\begin{proof}
  First, we consider the $1\le r<\infty$ case. It is sufficient to prove the $f\to_rg$ case. Let $W$ be a degenerate $\Gamma_r$ in $X\cup_AY$ which gives a $C_r$-homotopy from $f$ to $g$, and let $w_0,w_r$ be vertices of $W$ corresponding to the vertices $u_0,u_r$ of $\Gamma_r$. Since $\varphi_X$ is an $r$-cofibration, there is no directed path from $V(X)-V(A)$ to $V(\varphi_Y(A))$ in $X\cup_AY$. Hence if $w_0,w_r\in V(X)$, then $W\subset X\subset X\cup_AY$, and if $w_0,w_r\in V(Y)$, then $W\subset Y\subset X\cup_AY$. Hence in these cases, there is a copy of $W$ in $X\cup_A\widetilde{Y}$, implying $\tilde{f}_a\approx_r\tilde{g}_b$. The only remaining case is that $w_0\in V(Y)$ and $w_r\in V(X)-V(A)$. In this case, it follows from Lemma \ref{lift} that the two directed paths from $w_0$ to $w_r$ in $W$ lifts to $X\cup_A\widetilde{Y}$. Let $\widetilde{W}$ be the subgraph of $X\cup_A\widetilde{Y}$ defined by these lifts. Since $\varphi_X$ is an $r$-cofibration, we can argue as in the proof of Lemma \ref{lift} to obtain a subdivision of $\widetilde{W}$ in the following figure, where the arrows from $a$ to $\varphi_Y(a)$ and from $b$ to $\varphi_Y(b)$ are edges and the remaining arrows are directed paths.

  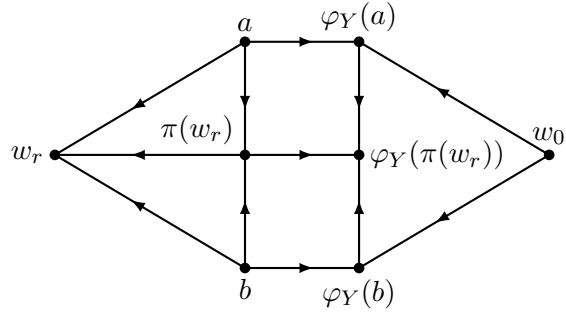
\begin{figure}[H]
    \label{subdivision W}
    \centering
    \begin{tikzpicture}[x=5mm, y=5mm, thick]
      \draw[myarrow=.6](5,0)--(0,0);
      \draw[myarrow=.6](5,0)--(8,0);
      \draw[myarrow=.6](5,3)--(5,0);
      \draw[myarrow=.6](5,3)--(0,0);
      \draw[myarrow=.6](5,3)--(8,3);
      \draw[myarrow=.6](8,3)--(8,0);
      \draw[myarrow=.6](5,-3)--(5,0);
      \draw[myarrow=.6](5,-3)--(0,0);
      \draw[myarrow=.6](5,-3)--(8,-3);
      \draw[myarrow=.6](8,-3)--(8,0);
      \draw[myarrow=.6](13,0)--(8,3);
      \draw[myarrow=.6](13,0)--(8,-3);
      \fill[black](0,0) node[left]{$w_r$} circle(2pt);
      \fill[black](5,0) node[above left]{$\pi(w_r)$} circle(2pt);
      \fill[black](8,0) node[right]{$\varphi_Y(\pi(w_r))$} circle(2pt);
      \fill[black](5,3) node[above]{$a$} circle(2pt);
      \fill[black](8,3) node[above]{$\varphi_Y(a)$} circle(2pt);
      \fill[black](5,-3) node[below]{$b$} circle(2pt);
      \fill[black](8,-3) node[below]{$\varphi_Y(b)$} circle(2pt);
      \fill[black](13,0) node[above]{$w_0$} circle(2pt);
    \end{tikzpicture}
    \caption{Subdivision of $\widetilde{W}$}
  \end{figure}

  \noindent As in the proof of Lemma \ref{lift}, the left two triangles and the middle two squares are degenerate $\Gamma_r$'s in $X\cup_A\widetilde{Y}$. Let $l(x,y)$ denote the length of the path from $x$ to $y$ in Figure \ref{subdivision W} for $(x,y)=(w_0,\varphi_Y(a)),(\varphi_Y(a),\varphi_Y(\pi(w_r)))$, $(a,\pi(w_r)),(a,w_r)$. Then since $\varphi_X$ is an $r$-cofibration, we have
  \begin{align*}
    l(w_0,\varphi_Y(a))+l(\varphi_Y(a),\varphi_Y(\pi(w_r)))&\le l(w_0,\varphi_Y(a))+l(a,\pi(w_r))\\
    &\le l(w_0,\varphi_Y(a))+l(a,w_r)\\
    &\le r.
  \end{align*}
  The same inequalities hold if we replace $a$ with $b$, so the right square in Figure \ref{subdivision W} is a degenerate $\Gamma_r$. Thus by Lemma \ref{subdivision} and the observation after it, we obtain $\tilde{f}_{a}\approx_r\tilde{g}_{b}$.

  Next, the $r=\infty$ case is proved verbatim by forgetting about the length of paths, completing the proof.
\end{proof}

Note that $V(X\cup_AY)=(V(X)-V(A))\sqcup V(Y)$ and $V(X\cup_A\widetilde{Y})=(V(X)-V(A))\sqcup V(Y)\sqcup V(A)$. Then we get an inclusion $V(X\cup_AY)\to V(X\cup_A\widetilde{Y})$. For $1\le r\le\infty$, we define a functor $F^{r}\colon\Pi_1^{r}(X\cup_AY)\to \Pi_1^{r}(X\cup_A\widetilde{Y})$ by the above inclusion on objects and
\[
  F^{r}([f])=
  \begin{cases}
    [\tilde{f}_a]&x\in V(X)-V(A)\text{ and }y\in V(Y)\\
    [f]&\text{otherwise}
  \end{cases}
\]
where $f$ is a directed path from $y$ to $x$ in $X\cup_AY$ and $\tilde{f}$ is as in Lemma \ref{lift}.

\begin{lemma}
  \label{F^r well-defined}
  For $1\le r\le\infty$, the functor $F^{r}$ is well-defined whenever $\varphi_X$ is an $r$-cofibration.
\end{lemma}

\begin{proof}
  Let $f$ is a directed path from $y$ to $x$ in $X\cup_AY$. Since there is no directed path from $V(A)$ to $V(X)-V(A)$, $f$ is either a path in $X,Y$ or a path from $y\in V(Y)$ to $x\in V(X)-V(A)$. Then the statement follows from Lemmas \ref{lift} and \ref{lift homotopy}.
\end{proof}

Now we are ready to prove:

\begin{theorem}
  \label{van Kampen pushout}
  If $\varphi_X\colon A\to X$ is an $r$-cofibration, then for $r\le s\le\infty$ the commutative diagram
  \[
    \xymatrix{
      \Pi_1^{s}(A)\ar[r]\ar[d]&\Pi_1^{s}(X)\ar[d]\\
      \Pi_1^{s}(Y)\ar[r]&\Pi_1^{s}(X\cup_AY)
    }
  \]
  is a pushout.
\end{theorem}

\begin{proof}
  By Lemma \ref{r-cofibration}, it is sufficient to prove the statement for $s=r$. We define a groupoid $\mathcal{P}^r$ by the pushout
  \[
    \xymatrix{
      \Pi_1^{r}(A)\ar[r]\ar[d]&\Pi_1^{r}(X)\ar[d]\\
      \Pi_1^{r}(Y)\ar[r]&\mathcal{P}^r.
    }
  \]
  Let $P^r\colon\mathcal{P}^r\to\Pi_1^{r}(X\cup_AY)$ be the natural functor, and let $Q^r\colon\Pi_1^{r}(X\cup_AY)\to\mathcal{P}^r$ be the composite of $F^{r}\colon\Pi_1^r(X\cup_AY)\to\Pi_1^r(X\cup_A\widetilde{Y})$ and the natural functor $\Pi_1^{r}(X\cup_A\widetilde{Y})\to\mathcal{P}^r$ which can be defined by Proposition \ref{pushout lift}. Let $e=(x,y)$ be an edge of $X\cup_AY$. If $x,y\in V(Y)$ or $x,y\in V(X)-V(A)$, then
  by the definition of $F^{r}$, we have
  \[
    P^r\circ Q^r([e])=[e]=Q^r\circ P^r([e]).
  \]
  If $x\in V(\varphi_Y(A))$ and $y\in V(X)-V(A)$, then by Lemma \ref{lift}, we get a path $\tilde{e}_a=\overline{(a,x)}\cdot(a,y)\colon\vec{I}_2(0)\to X\cup_A\widetilde{Y}$, where $\varphi_Y(a)=x$. Let $\hat{\rho}\colon\vec{I}_2(0)\to\vec{I}_1$ be the contraction of the edge $(1,0)$. Then $\tilde{\rho}\circ\tilde{e}_a=e\circ\hat{\rho}$, implying $\tilde{\rho}\circ\tilde{e}_a\approx_0e$. Thus we get
  \[
    Q^r\circ P^r([e])=[e]=P^r\circ Q^r([e])
  \]
  Since homomorphisms of $\mathcal{P}^r$ and $\Pi_1^{r}(X\cup_AY)$ are generated by edges of $X$ and $Y$, we obtain that $P^r$ and $Q^r$ are mutually inverse, completing the proof.
\end{proof}

Quite similarly to Corollaries \ref{van Kampen pi_1} and \ref{Mayer-vietoris sequence}, we can get the following corollaries to Theorem \ref{van Kampen pushout}.

\begin{corollary}
  \label{pushout fundamental group}
  If $\varphi_X$ is an $r$-cofibration and $A$ is connected, then for $r\le s\le\infty$ and $x_0\in V(A)$, the commutative square
  \[
    \xymatrix{
      \pi_1^{s}(A,x_0)\ar[r]\ar[d]&\pi_1^{s}(X,x_0)\ar[d]\\
      \pi_1^{s}(Y,\varphi_Y(x_0))\ar[r]&\pi_1^{s}(X\cup_AY,\varphi_Y(x_0))
    }
  \]
  is a pushout.
\end{corollary}

\begin{corollary}
  \label{Mayer-Vietoris sequence pushout}
  If $\varphi_X$ is an $r$-cofibration for $r\ge 2$ and $A,X,Y$ are connected, then for $r\le s<\infty$, the sequence
  \[
    E_{1,0}^s(A)\xrightarrow{((\varphi_X)_*,(\varphi_Y)_*)}E_{1,0}^s(X)\oplus E_{1,0}^s(Y)\xrightarrow{(j_X)_*-(j_Y)_*}E_{1,0}^s(X\cup_AY)\to 0
  \]
  is exact, where $j_X\colon X\to X\cup_AY$ and $j_Y\colon Y\to X\cup_AY$ denote the canonical maps.
\end{corollary}

As remarked above, a cofibration of Carranza \textit{et al.} \cite{CD} is exactly a $1$-cofibration. Then the above corollaries hold for any $2\le s<\infty$ whenever $\varphi_X$ is a cofibration. In particular, Corollary \ref{Mayer-Vietoris sequence pushout} is a partial answer to the question of Hepworth and Roff \cite{HR1}. We can also prove the Mayer-Vietoris sequence of $\mathrm{RH}_1$ from Corollary \ref{pushout fundamental group} under the condition that $E^r=E^{r+1}=\cdots=E^\infty$ for some $r$.

\end{document}